\documentclass[12pt, reqno, a4paper]{amsart}

\usepackage{amssymb}
\usepackage{hyperref}

\usepackage[english]{babel}

\theoremstyle{plain}
\newtheorem{theorem}{Theorem}[section]
\newtheorem{lemma}[theorem]{Lemma}
\newtheorem{proposition}[theorem]{Proposition}
\newtheorem{corollary}[theorem]{Corollary}
\newtheorem{conjecture}[theorem]{Conjecture}

\theoremstyle{remark}
\newtheorem{remark}[theorem]{Remark}

\theoremstyle{definition}
\newtheorem{example}[theorem]{Example}

\newtheorem{definition}[theorem]{Definition}

\newtheorem*{questions}{Questions}

\numberwithin{equation}{section}

\DeclareMathOperator{\Id}{Id}

\DeclareMathOperator{\id}{id}
\DeclareMathOperator{\ch}{char}

\DeclareMathOperator{\End}{End}
\DeclareMathOperator{\Aut}{Aut}

\DeclareMathOperator{\sign}{sign}
\DeclareMathOperator{\supp}{supp}
\DeclareMathOperator{\Alt}{Alt}

\DeclareMathOperator{\PIexp}{PIexp}

\begin{document}

\title[Actions of Ore extensions]{Actions of Ore extensions
and growth of polynomial $H$-identities}

\author{A.\,S.~Gordienko}
\address{Vrije Universiteit Brussel, Belgium}
\email{alexey.gordienko@vub.ac.be} 


\begin{abstract} 
We show that if $A$ is a finite dimensional associative $H$-module algebra for an arbitrary Hopf algebra $H$, then the proof of the analog of Amitsur's conjecture for $H$-codimensions of $A$ can be reduced to the case when $A$ is $H$-simple. (Here we do not require that the Jacobson radical of $A$ is an $H$-submodule.)
 As an application, we prove that if $A$ is a finite dimensional associative $H$-module algebra where $H$ is a Hopf algebra $H$ over a field of characteristic $0$ such that $H$ is constructed by an iterated Ore extension of a finite dimensional semisimple Hopf algebra by skew-primitive elements (e.g. $H$ is a Taft algebra), then there exists integer $\PIexp^H(A)$. In order to prove this, we study the structure of algebras simple with respect to an action of an Ore extension.
\end{abstract}


\thanks{Supported by Fonds Wetenschappelijk Onderzoek~--- Vlaanderen post doctoral fellowship (Belgium).}

\maketitle

\section{Introduction}

Numerical sequences attached to an object are used in many areas of algebra. One of such sequences is the sequence of codimensions $c_n(A)$ of polynomial identities of an algebra $A$. By the definition, $c_n(A) := \dim \frac{P_n}{P_n\cap \Id(A)}$ where $P_n$ is the vector space of multilinear
polynomials in the non-commuting variables $x_1,\ldots, x_n$ and $\Id(A)$ is the set of polynomial identities of $A$, i.e. such polynomials that vanish under all evaluations in $A$. 

Codimensions arise naturally when one calculates a basis for polynomial identities of an algebra over a field of characteristic $0$. In addition, codimesions were used by A.~Regev~\cite{RegevTensor} to show that the tensor product of PI-algebras (i.e. algebras satisfying a non-trivial polynomial identity) is again a PI-algebra.

In 1999 A.~Giambruno and M.\,V.~Zaicev~\cite{GiaZai99} proved the following conjecture:
\begin{conjecture}[S.\,A.~Amitsur]
For every associative PI-algebra $A$ over a field $F$ of characteristic $0$ there exists $$\PIexp(A):=\lim_{n\to\infty}\sqrt[n]{c_n(A)}\in\mathbb Z_+.$$
\end{conjecture}
 Probably, the most exciting thing in their proof was that they provided (at least in the case when $A$ is finite dimensional and $F$ is algebraically closed)
the formula for $\PIexp(A)$ involving the Jacobson radical $J(A)$ of $A$ and simple components of $A/J(A)$.
In 2002 M.\,V.~Zaicev~\cite{ZaiLie} proved the analog of Amitsur's conjecture for finite dimensional
Lie algebras. The formula for the PI-exponent of Lie algebras was more complicated, than in the associative case, and involved annihilators of irreducible factors of the adjoint representation of the Lie algebras.

The results mentioned above suggest a further investigation of the connection between the structure of an algebra and the asymptotic behaviour of its codimensions, especially in the case when the algebra is endowed with an additional structure, e.g. a grading, an action of a group $G$ by automorphisms and anti-automorphisms, an action of a Lie algebra by derivations, since in this case it is reasonable to consider, respectively, graded, $G$- or differential identities~\cite{BahtGiaZai, BahtDrensky, Kharchenko}.
The case of a Hopf algebra action is of special interest because of the recent developments in the theory of algebras with Hopf algebra actions~\cite{CuadraEtingofWalton, EtingofWalton}. 

In order to study all these cases simultaneously, it is useful to
consider generalized $H$-actions (see the definition in Section~\ref{SectionHModGen}). Probably, the first who
used such actions and studied polynomial $H$-identities was Allan Berele~\cite[remark before Theorem~15]{BereleHopf} in 1996 (see also~\cite{BahturinLinchenko}).

Denote by $(c_n^H(A))_{n=1}^{\infty}$ the sequence of codimensions of polynomial $H$-identities of an
algebra $A$ with a generalized $H$-action. The following informal questions arise naturally:

\begin{questions}
Under which conditions on $H$-action holds
the analog of the Amitsur's conjecture,
i.e. there still exists an integer $$\PIexp^H(A):=\lim_{n\to\infty}\sqrt[n]{c_n(A)}\quad ?$$  What do we need from the structure of $A$ ?
\end{questions}

Note that the analog of the Amitsur's conjecture itself becomes here in some sense a tool for studying and understanding $H$-actions.

Example~\ref{ExampleInfCodim} below shows that in the case when $A$ and $H$ are both infinite dimensional the codimensions $c_n^H(A)$ could be infinite. Therefore it is reasonable to reduce our consideration to the case when either $A$ or $H$ is finite dimensional.

In the case of associative PI-algebras graded by a finite group,
the analog of Amitsur's conjecture
for codimensions of graded identities was proved in 2010--2011 by
E.~Aljadeff,  A.~Giambruno, and D.~La~Mattina~\cite{AljaGia, AljaGiaLa, GiaLa}. Their proof of the lower bound was based on the classification of finite dimensional graded-simple algebras. 

In 2013 the author proved~\cite{ASGordienko8} the analog of Amitsur's conjecture for finite dimensional associative algebras $A$ with a generalized $H$-action, where $H$ is an associative algebra with $1$, such that the Jacobson radical $J(A)$ is an $H$-submodule and $A/J(A)$ is a direct sum of $H$-simple algebras. That led to the proof of the analog of Amitsur's conjecture for finite dimensional associative algebras graded by any group
not necessarily finite~\cite{ASGordienko9}. 
In 2015 Yaakov Karasik~\cite{Karasik} proved the existence of integer $\PIexp^H(A)$ for (not necessarily
finite dimensional) $H$-module PI-algebras $A$ for finite dimensional semisimple Hopf algebras $H$.

In 2014 the author~\cite{ASGordienko13} (see also~\cite{GordienkoJanssensJespers}) found the first counterexamples to the analog of Amitsur's conjecture for $A$ in the case when $A$ is a module algebra for a finite dimensional bialgebra $H$. 

Despite the counterexamples mentioned above, we still believe that the analog of Amitsur's conjecture 
is true in the following form which belongs to Yu.\,A.~Bahturin:

\begin{conjecture}\label{ConjectureAmitsurBahturin} Let $A$ be a finite dimensional associative $H$-module algebra
for a Hopf algebra $H$ over a field of characteristic $0$. Then there exists
an integer $\PIexp^H(A):=\lim\limits_{n\to\infty}
 \sqrt[n]{c^H_n(A)}$.
\end{conjecture}

As we have mentioned above, the proofs in all previous papers~\cite{AljaGia, AljaGiaLa, GiaLa, ASGordienko8, ASGordienko11, ASGordienko12, Karasik} worked only in the case when the Jacobson radical $J(A)$ was an $H$-submodule or $A$ was $H$-simple itself.
In the current article we do not assume that the Jacobson radical of $A$ is $H$-invariant,
replacing the Wedderburn~--- Mal'cev theorem by its weak analog (Lemma~\ref{LemmaWeakWedderburnMalcevHRad}) which still makes it possible to transfer the computations to $H$-simple algebras. In Theorem~\ref{TheoremHmodHRadAmitsurPIexpHBdimB} we show that Conjecture~\ref{ConjectureAmitsurBahturin}
can be reduced to the existence of an $H$-polynomial which is a polynomial non-identity
and has sufficiently many alternations. More precisely, we prove that if $H$ is a unital associative algebra, $A$ is a finite dimensional associative algebra with a generalized $H$-action, $J^H(A)$ is its maximal nilpotent $H$-invariant ideal, and $A/J^H(A)$ is a direct sum of $H$-simple algebras satisfying Property (*) (see Section~\ref{SectionPropertyStar}), then there exists integer $\PIexp^H(A)$. In other words, the analog of Amitsur's conjecture is a consequence of Property (*) and the $H$-invariant analog of the Wedderburn~--- Artin Theorem.

 The $H$-invariant analog of the Wedderburn~--- Artin Theorem holds for all finite dimensional $H$-module algebras for any Hopf algebra~$H$ \cite[Theorem 1.1]{SkryabinHdecomp}, \cite[Lemma 4.2]{SkryabinVanOystaeyen} (see also Lemma~\ref{LemmaHSemiSimpleIsUnital} below), however there exist finite dimensional algebras with a generalized $H$-action, namely semigroup graded algebras, where it does not hold~\cite[Example~4]{ASGordienko13}.
   
   Property (*) implies that the exponents of growth of $H$-identities of the corresponding $H$-simple algebras are integer and equal their dimensions. Property (*) holds for all finite dimensional semisimple (in ordinary sense) $H$-simple algebras~\cite[Theorem~7]{ASGordienko3} (see also the remark in Example~\ref{ExampleHSemiSimplePropertyStar} below) and for all finite dimensional algebras simple with respect to an action of a Taft algebra $H_{m^2}(\zeta)$~\cite[Lemma~7]{ASGordienko12}. However, there exist finite dimensional algebras
   with a generalized $H$-action, which again are semigroup graded algebras, where the $H$-PI-exponent is non-integer and Property (*) does not hold~\cite{ASGordienko13}. (See also~\cite{GordienkoJanssensJespers}.)

As an application of Theorem~\ref{TheoremHmodHRadAmitsurPIexpHBdimB}, in Corollary~\ref{CorollaryHOreAmitsur} we show that Conjecture~\ref{ConjectureAmitsurBahturin} holds for the class $\mathcal C$ of all Hopf algebras $H$ that are constructed by an iterated Ore extension of a finite dimensional semisimple Hopf algebra $\tilde H$ by skew-primitive elements (see the precise
definition of such algebras $H$ in Definition~\ref{DefinitionOreConstructed}), which is a crucial step for proving Conjecture~\ref{ConjectureAmitsurBahturin} for an arbitrary pointed Hopf algebra.

The class $\mathcal C$ is rather large. It includes Taft algebras $H_{m^2}(\zeta)$ as well as Hopf algebras $H(C,n,c,c^*,a,b)$ (see~\cite[Definition 5.6.15]{Danara}) that were used to answer in the negative Kaplansky's conjecture on the finiteness of the isomorphism
types of Hopf algebras of a given finite dimension~\cite{AndrusSchneider, BDG1, BDG2}. Furthermore, the class $\mathcal C$ contains non-pointed Hopf algebras,
e.g. two Hopf algebras of dimension $16$ (see~\cite[Theorem~5.1]{CalDascMasMen}).

Finite dimensional $H_{m^2}(\zeta)$-module algebras that contain no nonzero nilpotent elements were classified in~\cite[Theorem~2.5]{MontgomerySchneider}. Exact module categories over the category $\mathrm{Rep}(H_{m^2}(\zeta))$ were studied in~\cite[Theorem~4.10]{EtingofOstrik}.
Finite dimensional $H_{m^2}(\zeta)$-simple $H_{m^2}(\zeta)$-module algebras were classified in~\cite{ASGordienko11, ASGordienko12}.
$H_{m^2}(\zeta)$-actions on path algebras of quivers were studied in~\cite{KinserWalton}.

In Theorem~\ref{TheoremHSimpleOreExtStructure} we show that for every $H$-simple $H$-module algebra $A$, where $H\in \mathcal C$,
the $\tilde H$-module algebra $A/J^{\tilde H}(A)$ is $\tilde H$-simple and if $J^{\tilde H}(A)\ne 0$,
then $A=\bigoplus_{k=0}^{m-1} v^k\tilde J$ for some minimal $\tilde H$-invariant ideal 
$\tilde J \subseteq J^{\tilde H}(A)$ and the restriction of the natural surjective
homomorphism $A \twoheadrightarrow A/J^{\tilde H}(A)$ on $v^{m-1}\tilde J$ is an $F$-linear bijection
$v^{m-1}\tilde J \mathrel{\widetilde{\to}} A/J^{\tilde H}(A)$.
This enables to show in Theorem~\ref{TheoremHSimpleOreExtStructure} that Property (*) for $\tilde H$-simple algebras implies Property (*) for $H$-simple algebras, which we use in Corollary~\ref{CorollaryHOreAmitsur}.

Imposing two additional restrictions which both hold in all the examples of Hopf algebras $H$ above, namely, \begin{enumerate}\item $\varphi(g)=\zeta g$ where $g$ is defined by $\Delta v = g \otimes v + v \otimes 1$
and $\zeta$ is a $t$th primitive root of unity, $t\in\mathbb N$,
\item  $v^t \tilde J \subseteq \tilde J$ (e.g. $v^t \in \tilde H$),
\end{enumerate}
 we prove in Theorem~\ref{TheoremHSimpleOreExtMul}
the formula for the multiplication in $A$ involving the multiplication in $A/J^{\tilde H}(A)$.

 \section{$H$-module algebras and their generalizations}\label{SectionHModAndGen}
 \subsection{$H$-module algebras}\label{SectionHmoduleAlgebras}

 An algebra $A$
over a field $F$
is a \textit{(left) $H$-module algebra}
for some Hopf algebra $H$
if $A$ is endowed with a structure of a (left) $H$-module such that
$h(ab)=(h_{(1)}a)(h_{(2)}b)$
for all $h \in H$, $a,b \in A$. Here we use Sweedler's notation
$\Delta h = h_{(1)} \otimes h_{(2)}$ where $\Delta$ is the comultiplication
in $H$ and the sign of the sum is omitted.
(Note that we do not require from $A$ to be unital.)

We say that an $H$-module algebra $A$ is a \textit{unital $H$-module algebra}
if there exists $1_A$ and, in addition, $h 1_A = \varepsilon(h) 1_A$ for all $h\in H$.

We refer the reader to~\cite{Danara, Montgomery, Sweedler}
   for an account
  of Hopf algebras and algebras with Hopf algebra actions.

\subsection{Algebras with a generalized $H$-action}\label{SectionHModGen}
  
In the most of the constructions related to polynomial
identities themselves the restrictions in the definition of an $H$-module algebra are unnecessary, and we can consider a more general notion
of an algebra with a generalized $H$-action.

  Let $H$ be an arbitrary associative algebra with $1$ over a field $F$.
We say that an algebra $A$ is an algebra with a \textit{generalized $H$-action}
if $A$ is endowed with a homomorphism $H \to \End_F(A)$
and for every $h \in H$ there exist some $k\in \mathbb N$ and some $h'_i, h''_i, h'''_i, h''''_i \in H$, $1\leqslant i \leqslant k$,
such that 
\begin{equation}\label{EqGeneralizedHopf}
h(ab)=\sum_{i=1}^k\bigl((h'_i a)(h''_i b) + (h'''_i b)(h''''_i a)\bigr) \text{ for all } a,b \in A.
\end{equation}

Equivalently, there exist linear maps $\Delta, \Theta \colon H \to H\otimes H$ (not necessarily coassociative)
such that 
$$ h(ab)=\sum\bigl((h_{(1)} a)(h_{(2)} b) + (h_{[1]} b)(h_{[2]} a)\bigr) \text{ for all } a,b \in A.$$ (Here we use the notation $\Delta(h)= \sum h_{(1)} \otimes h_{(2)}$ and $\Theta(h)= \sum  h_{[1]} \otimes h_{[2]}$.)

\begin{example}\label{ExampleHmodule} If $A$ is an $H$-module algebra,
then $A$ is an algebra with a generalized $H$-action.
\end{example}

\begin{example}\label{ExampleIdFT}
Recall that if $T$ is a semigroup, then the \textit{semigroup algebra} $FT$ over a field is the vector space with the formal basis $(t)_{t\in T}$ and the multiplication induced by the one in $T$.
Let $A$ be an associative algebra with an action 
of a semigroup $T$ by endomorphisms and anti-endomorphisms. Then $A$ is an algebra with
 a generalized $FT$-action.
\end{example}

\begin{example}\label{ExampleIdGr}
Let $A=\bigoplus_{t\in T} A^{(t)}$ be a \textit{graded} algebra for some set of indices $T$, i.e. for every $s,t \in T$ there exists $r\in T$
such that $A^{(s)}A^{(t)}\subseteq A^{(r)}$. Denote this grading by $\Gamma$. Note that $\Gamma$ defines on $T$
a partial operation $\star$ with the domain $T_0:=\lbrace (s,t) \mid A^{(s)}A^{(t)} \ne 0 \rbrace$ by $s\star t = r$.
Consider the algebra $F^T$ of functions from $T$ to $F$.
Then $F^T$ acts on $A$ naturally: $ha = h(t)a$ for all $a\in A^{(t)}$.
Let $h_t(s):=\left\lbrace\begin{smallmatrix} 1 & \text{if} & s=t,\\ 0 & \text{if} & s\ne t.\end{smallmatrix} \right.$
If the \textit{support} $$\supp \Gamma := \lbrace t\in T \mid A^{(t)}\ne 0\rbrace$$ of $\Gamma$ is finite, $T_0$ is finite too and we have
\begin{equation}\label{EqIdentityHFiniteSupp}h_r(ab)=\sum_{\substack{(s,t)\in T_0,\\ r=s\star t}}
h_s(a)h_t(b). \end{equation}
(Since the expression is linear in $a$ and $b$, it is sufficient to check it only for homogeneous $a,b$.)
 Note that $(h_t)_{t\in T}$ is a basis in $F^T$. Again by the linearity
 we get~(\ref{EqGeneralizedHopf}) for every $h\in F^T$, and $A$ is an algebra
with a generalized $F^T$-action.
\end{example}

Let $A$ be an algebra with a generalized $H$-action for some associative algebra $H$ with $1$ over a field $F$. We say that a subspace $V\subseteq A$ is  \textit{invariant} under the $H$-action if $HV=V$, i.e. $V$ is an $H$-submodule. Denote by $J^H(A)$ the maximal nilpotent $H$-invariant two-sided ideal of $A$. We call $J^H(A)$ the \textit{$H$-radical} of $A$. If $A$ is Noetherian (e.g. Artinian or even finite-dimensional), then $J^H(A)$ always exists. If $A^2\ne 0$ and $A$ has no non-trivial two-sided $H$-invariant ideals, we say that $A$ is \textit{$H$-simple}.

\subsection{$H$-invariant Wedderburn~--- Artin theorem}

  In~\cite[Theorem 1.1]{SkryabinHdecomp}, \cite[Lemma 4.2]{SkryabinVanOystaeyen}
Sergey Skryabin and Freddy Van Oystaeyen proved
the following theorem.

\begin{theorem}[S.\,M.~Skryabin~--- F.~Van Oystaeyen]\label{TheoremSkryabinVanOystaeyen}
Let $A$ be a unital left $H$-module right Artinian associative algebra for some Hopf algebra over a field $F$
such that $J^H(A)=0$.
Then $A=B_1 \oplus \ldots \oplus B_q$ (direct sum of $H$-invariant ideals) for some $q\in \mathbb Z_+$ and some $H$-simple $H$-module algebras $B_i$.
\end{theorem}

In order to apply Theorem~\ref{TheoremSkryabinVanOystaeyen}
without any assumptions on the existence of $1_A$, we make the following observation which is also of independent interest.

\begin{lemma}\label{LemmaHSemiSimpleIsUnital}
Let $A$ be a left $H$-module right Artinian associative algebra for some Hopf algebra over a field $F$
such that $J^H(A)=0$.
Then $A$ is a unital $H$-module algebra.
\end{lemma}
\begin{proof}
Let $A^+ := A \oplus F1_{A^+}$ where $1_{A^+}$ is the adjoint unit of $A$.
Define $h 1_{A^+} := \varepsilon(h)1_{A^+}$ for all $h\in H$.
Then $A^+$ is a unital $H$-module algebra and $A$ is a two-sided $H$-invariant
ideal of $A^+$. Since $A^+/A$ and $A$ are right Artinian $A^+$-modules, $A^+$ is a right Artinian algebra
and we can apply Theorem~\ref{TheoremSkryabinVanOystaeyen}.
We get $A^+=B_1 \oplus \ldots \oplus B_q$ (direct sum of $H$-invariant ideals) for some $q\in \mathbb Z_+$ and some unital $H$-simple $H$-module algebras $B_i$.
Since $A$ is a two-sided $H$-invariant ideal of $A^+$
and each $AB_i$ equals either $0$ or $B_i$, we obtain that $A$ is the direct sum
of all but one $B_i$. Since all $B_i$ are unital and $h 1_{B_i} = \varepsilon(h)
1_{B_i}$ for $h\in H$, we get the proposition.
\end{proof}

As a consequence, if $A$ is a finite dimensional $H$-module associative algebra for a Hopf algebra $H$,
we always have $A/J^H(A) = B_1 \oplus B_2 \oplus \ldots
\oplus B_q$ (direct sum of $H$-invariant ideals) for some $H$-simple $H$-module algebras $B_i$.

\subsection{Weak Wedderburn~--- Mal'cev theorem for associative algebras with a generalized $H$-action}

A strong Wedderburn~--- Mal'cev theorem for associative algebras with a generalized $H$-action would say that for a finite dimensional associative algebra $A$ with a generalized $H$-action there exists
a homomorphism $\varkappa \colon A/J^H(A) \hookrightarrow A$ of algebras and $H$-modules such that $\pi\varkappa = \id_{A/J^H(A)}$
where $\pi \colon A \twoheadrightarrow A/J^H(A)$ is the natural surjective homomorphism. Unfortunately, this is not always true even for $H$-module algebras, see~\cite[Example~4]{ASGordienko8}.

Here we prove a weak form of the Wedderburn~--- Mal'cev theorem, namely,
that there exists an $F$-linear map $\varkappa \colon A/J^H(A) \hookrightarrow A$ such that $\pi\varkappa = \id_{A/J^H(A)}$ and $\varkappa$ is a $(B,B)$-bimodule homomorphism for a maximal semisimple subalgebra 
$B \subseteq A/J^H(A)$. Recall that by $J(A)$ we denote the ordinary Jacobson radical of an algebra $A$.

\begin{lemma}\label{LemmaWeakWedderburnMalcevHRad}
Let $A$ be a finite dimensional associative algebra with a generalized $H$-action for some associative algebra $H$ with $1$ over an algebraically closed field $F$. Denote by $\pi$ the natural surjective homomorphism $A \twoheadrightarrow A/J^H(A)$. There exists an $F$-linear embedding $\varkappa \colon A/J^H(A) \hookrightarrow A$ such that $\pi\varkappa=\id_{A/J^H(A)}$
and for some maximal semisimple (in the ordinary sense) subalgebra $B \subseteq A/J^H(A)$, where $A/J^H(A) = B \oplus J(A/J^H(A))$ (direct sum of subspaces), we have  $\varkappa(ba)=\varkappa(b)\varkappa(a)$, $\varkappa(ab)=\varkappa(a)\varkappa(b)$ for all $a\in A/J^H(A)$ and $b\in B$.
\end{lemma}
\begin{proof}
By the ordinary Wedderburn~--- Mal'cev theorem,
there exists a
maximal semisimple subalgebra $B_0 \subseteq A$ such that $A = B_0 \oplus J(A)$ (direct sum of subspaces).
Now we treat $A$ as a $(B_0, B_0)$-bimodule. We claim that $A$ is a direct sum of irreducible $(B_0, B_0)$-subbimodules.  

Since $B_0$ is a semisimple algebra over an algebraically closed field,
the algebra $B_0 \otimes B_0^\mathrm{op}$ is semisimple too,
and $A$ is a completely reducible
left $B_0 \otimes B_0^\mathrm{op}$-module where $B_0^\mathrm{op}$ is anti-isomorphic
to $B_0$ and $(b_1\otimes b_2)a :=b_1 a b_2$ for all $b_1 \otimes b_2 \in B_0 \otimes B_0^\mathrm{op}$
and $a\in A$. Since we do not require from $A$ to be unital, this implies only that
$1_{B_0} A\ 1_{B_0}$ is a completely reducible $(B_0, B_0)$-bimodule.

Consider the Pierce decomposition $$A=(1-1_{B_0})A(1-1_{B_0})\oplus 1_{B_0}A(1-1_{B_0})\oplus
(1-1_{B_0})A\ 1_{B_0}\oplus 1_{B_0} A\ 1_{B_0}$$ (direct sum of $(B_0, B_0)$-subbimodules)
where $1$ is the formal unity.
Here $1_{B_0}A(1-1_{B_0})$ is a completely reducible left $B_0$-module,
$(1-1_{B_0})A\ 1_{B_0}$ is a completely reducible right $B_0$-module,
and $(1-1_{B_0})A(1-1_{B_0})$ is a vector space with zero $(B_0, B_0)$-action.
Hence $A$ is a direct sum of irreducible $(B_0, B_0)$-bimodules.
Therefore, there exists a $(B_0, B_0)$-subbimodule $N\subseteq A$
such that $J(A)=N\oplus J^H(A)$. 
Note that $$\pi\bigr|_{(B_0\oplus N)} \colon (B_0\oplus N) \mathrel{\widetilde\to}
A/J^H(A)$$ is an $F$-linear bijection.
Define $\varkappa := \left(\pi\bigr|_{(B_0\oplus N)}\right)^{-1}$.
Let $B := \pi(B_0)$. 
Then $A/J^H(A) = \pi(B_0) \oplus \pi(J(A))=B \oplus J(A/J^H(A))$.

Let $a \in A/J^H(A)$ and $b\in B$. Then
$\pi(\varkappa(ab))=ab=\pi\varkappa(a)\pi\varkappa(b)=\pi(\varkappa(a)\varkappa(b))$.
Since both $\varkappa(ab), \varkappa(a)\varkappa(b) \in B_0\oplus N$,
we get $\varkappa(ab)=\varkappa(a)\varkappa(b)$.
Analogously, $\varkappa(ba)=\varkappa(b)\varkappa(a)$.
\end{proof}

\section{$H$-simple algebras and their $\tilde H$-invariant ideals}\label{SectionHSimpleOreExtStructure}

In this section we show that if $A$ is a finite dimensional $H$-simple $H$-module algebra
where $H$ is a Hopf algebra which is constructed by an Ore extension of a Hopf algebra $\tilde H$
by a skew-primitive element $v$, then $A/J^{\tilde H}(A)$ is $\tilde H$-simple and if 
$A/J^{\tilde H}(A) \ne 0$, then, applying $v$ several times, we can always
map a given nonzero element $a$ to an element which is nonzero modulo $J^{\tilde H}(A)$.
This will be used later in the proof of Theorem~\ref{TheoremHOrePropertyStar}.

We first recall the definition of an Ore extension.

\begin{definition}
Let $C$ be an associative algebra over a field $F$, let
 $\varphi \colon C\mathrel{\widetilde{\to}}C$
be an algebra automorphism, and
let $\delta \colon C \to C$ be a \textit{$\varphi$-skew-derivation}, i.e.
$\delta(ab)=\varphi(a)\delta(b)+\delta(a)b$ for all $a,b\in C$.
Then the \textit{Ore extension} $C[v,\varphi,\delta]$
consists of formal polynomials $\sum_{i=0}^n a_i v^n$, $a_i\in C$, $n\in\mathbb Z_+$,
where the multiplication is induced by that one in $A$ together with the relations
$va - \varphi(a) v = \delta(a)$, $a\in C$.
\end{definition}

Note that if $B$ is an algebra generated by an element $v\in B$ and a subalgebra $C \subseteq B$
and $\varphi \colon C\mathrel{\widetilde{\to}}C$
is an algebra automorphism such that $v a - \varphi(a)v \in C$ for all $a\in C$,
then the map $$\delta(a)=v a - \varphi(a)v,\ a\in C$$ is always a $\varphi$-skew-derivation on $C$.
In this case $B \cong C[v,\varphi,\delta]/I$ where $I$ is an ideal of $B$ such that $C \cap I = 0$.

\begin{definition}\label{DefinitionOreConstructed}
We say that $B$ is \textit{constructed by an Ore extension of $C$ by the element $v$}.
\end{definition}

Recall that an element $g$ of a Hopf algebra $H$ is \textit{group-like}
if $\Delta g = g\otimes g$ and $g\ne 0$. The set $G(H)$ of all group-like elements of $H$ is a group under the multiplication in $H$. An element $v\in H$ is \textit{skew-primitive}
if $\Delta v = g_1 \otimes v + v \otimes g_2$ for some 
$g_1, g_2 \in G(H)$.

\begin{lemma}\label{LemmaHOreRadical}
Let $H$ be a Hopf algebra over a field $F$ generated as an algebra by a subHopfalgebra $\tilde H$ and a skew-primitive element $v\in H$, $\Delta v = g_1 \otimes v + v \otimes g_2$, $g_1, g_2 \in G(\tilde H)$.
Let $A$ be a finite dimensional $H$-simple $H$-module algebra. Then either $A$ is $\tilde H$-simple
or $J^{\tilde H}(A)\ne 0$.
\end{lemma}
\begin{proof}
Suppose $J^{\tilde H}(A)= 0$.
By Theorem~\ref{TheoremSkryabinVanOystaeyen},
$A=N_1 \oplus \ldots \oplus N_s$ (direct sum of $\tilde H$-invariant ideals)
for some unital $\tilde H$-simple algebras $N_i$. 
We claim that all $N_i$ are $H$-invariant. It is sufficient to show that
$vN_i\subseteq N_i$ for all $1\leqslant i \leqslant s$.
Let $a\in N_i$ and $b\in N_j$, $i\ne j$.
Then $$(va)b=v(a (g_2^{-1}b))-(g_1a)(vg_2^{-1}b)=-(g_1 a)(vg_2^{-1}b)
\in N_i.$$ At the same time $b\in N_j$ and $(va)b \in N_j$.
Hence $(va)b=0$. Analogously, $b(va)=0$. Since $b\in N_j$, $j\ne i$, is arbitrary, we get $va \in N_i$. Hence all $N_i$ are $H$-invariant,
$s=1$, and $A$ is $\tilde H$-simple.
\end{proof}

\begin{theorem}\label{TheoremHSimpleOreExtStructure}
Let $H$ be a Hopf algebra over a field $F$ generated as an algebra by a subHopfalgebra $\tilde H$ and a skew-primitive element $v\in H$, $\Delta v = g_1 \otimes v + v \otimes g_2$, $g_1, g_2 \in G(\tilde H)$.
Also suppose that there exists an algebra automorphism $\varphi \colon \tilde H \mathrel{\widetilde{\to}} \tilde H$
such that $vh-\varphi(h)v\in \tilde H$ for all $h\in\tilde H$.
Let $A$ be a finite dimensional $H$-simple $H$-module algebra with $J^{\tilde H}(A)\ne 0$. 
Denote by $p$ the natural number such that 
$J^{\tilde H}(A)^p=0$ and $J^{\tilde H}(A)^{p-1}\ne 0$.
Let $\tilde J \subseteq J^{\tilde H}(A)^{p-1}$ be a minimal
two-sided $\tilde H$-invariant ideal.
Then there exists $m\in\mathbb N$
such that
 $$I_k := \bigoplus_{i=0}^{k-1} v^i \tilde J$$ are two-sided $\tilde H$-invariant ideals
 for $1 \leqslant k \leqslant m$,
$I_m = A$, $I_{m-1} = J^{\tilde H}(A)$.
Moreover, the map $v^k \tilde J \twoheadrightarrow v^{k+1}\tilde J$
defined by $a \mapsto va$ is an $F$-linear bijection
for all $0 \leqslant k \leqslant m-1$.
In addition, $A/J^{\tilde H}(A)$ is an $\tilde H$-simple algebra.
\end{theorem}

Before proving Theorem~\ref{TheoremHSimpleOreExtStructure},
we give several definitions.

Let $A$ be a (left) $H$-module algebra. We say that $M$ is a \textit{$(H,(A,A))$-bimodule}
if $M$ is a left $H$-module, an $(A,A)$-bimodule, and $h(am)=(h_{(1)}a)(h_{(2)}m)$,
and $h(ma)=(h_{(1)}m)(h_{(2)}a)$ for all $a\in A$ and $m\in M$.

Let $\psi \colon M_1 \to M_2$ be an $F$-linear map where $M_1$ and $M_2$
are $(H,(A,A))$-bimodules. Let $g_1, g_2 \in G(H)$
and let $\varphi$ be an algebra automorphism of $H$. We say that $\psi$ is a \textit{$(\varphi, g_1, g_2)$-homomorphism} if $\psi(am)=(g_1 a)\psi(m)$, $\psi(ma)=\psi(m) (g_2 a)$,
and $\psi(hm)=\varphi(h)\psi(m)$.
If $\psi$ is bijective, we say that $\psi$ is a
\textit{$(\varphi, g_1, g_2)$-isomorphism}.

\begin{proof}[Proof of Theorem~\ref{TheoremHSimpleOreExtStructure}.]
If $\tilde J$ is invariant under the action of $v$,
then $\tilde J$ is an $H$-invariant two-sided ideal of $A$ and we get $A=\tilde J$ since $A$ is $H$-simple. Now $\tilde J\subseteq J(A)$ implies $A^2=\tilde J^2 = 0$ which contradicts to the definition of an $H$-simple algebra.
Hence $v \tilde J \subsetneqq \tilde J$.

Let $I_k := \sum_{i=0}^{k-1} v^i \tilde J$, $k\geqslant 1$.
Since $hv a \in v\varphi^{-1}(h)a + \tilde H a$
for every $h\in \tilde H$ and $a\in A$,
each $I_k$ is $\tilde H$-invariant.

Moreover, since $v$ is skew-primitive,
we have $$b(va) = v ((g_1^{-1} b)a) - (v g_1^{-1} b)(g_2 a)$$
and $$(va)b = v(a(g_2^{-1}b)) - (g_1a)(vg_2^{-1}b)$$ for all $a,b \in A$.
Taking $a\in I_k$ we obtain by induction that all $I_k$ are two-sided ideals.

Since $A$ is finite-dimensional, there exists $m\in\mathbb N$
such that $I_k \subsetneqq I_{k+1}$ for all $1\leqslant k \leqslant m-1$
and $I_{m+1}=I_m$. Then $I_m$ is invariant under the action of $v$
and we have $A=I_m$ since $A$ is $H$-simple.

Let $I_0:=0$. Define the $F$-linear maps $\psi_k \colon I_k/I_{k-1} \twoheadrightarrow I_{k+1}/I_k$, $1 \leqslant k \leqslant m-1$, by $$\psi_k(a + I_{k-1}):= va+ I_k,\quad a\in I_k.$$

Note that $$\psi_k((a+I_{k-1})b)=v(ab)+ I_k=(g_1 a)(vb)+(v a)(g_2b)+ I_k =\psi_k(a+I_{k-1})(g_2b)$$
and $$\psi_k(b(a+I_{k-1}))= v(ba)+ I_k
=(g_1 b)(va)+(v b) (g_2a)+ I_k =
(g_1 b)\psi_k(a+I_{k-1})$$
for all $a\in I_k$, $b\in A$.

Moreover $$\psi_k(h(a+I_{k-1}))=vha + I_k=
\varphi(h) va +I_k
= \varphi(h)\psi_k(a+I_{k-1}).
$$
Therefore, $\psi_k$ is an $(\varphi,g_1,g_2)$-homomorphism.

Since $I_1=\tilde J$ is an irreducible $(\tilde H,(A,A))$-bimodule,
either $\psi_1$ is bijective or $I_2/I_1=0$. i.e. $m=1$.
In the first case $I_2/I_1$ is again an irreducible $(\tilde H,(A,A))$-bimodule.
Continuing this procedure we obtain that $I_k/I_{k-1}$, $1\leqslant k \leqslant m$, are irreducible $(\tilde H,(A,A))$-bimodules $(\varphi^{1-k},g_1^{1-k},g_2^{1-k})$-isomorphic to~$\tilde J$. Comparing their dimensions, we get
that the map $v^k \tilde J \twoheadrightarrow v^{k+1}\tilde J$
defined by $a \mapsto va$ must be an $F$-linear bijection
for all $0 \leqslant k \leqslant m-1$ and the sum in the definition of $I_k$ is direct

The $(\tilde H,(A,A))$-bimodule version
of the Jordan~--- H\"older theorem implies that in any composition
series of $(\tilde H,(A,A))$-bimodules in $A$
each factor is $(\tilde \varphi,\tilde g_1, \tilde g_2)$-isomorphic
to $\tilde J$ for appropriate $\tilde\varphi \in \Aut(\tilde H)$, $\tilde g_1, \tilde g_2\in G(\tilde H)$. 
By Theorem~\ref{TheoremSkryabinVanOystaeyen},
$A/J^{\tilde H}(A)=N_1 \oplus \ldots \oplus N_s$ (direct sum of $\tilde H$-invariant ideals)
for some $\tilde H$-simple algebras $N_i$. 
Suppose $s\geqslant 2$. Since $N_1$ and $N_2$ are irreducible factors
in an $(\tilde H,(A,A))$-bimodule composition series of $A$,
there exists a $(\tilde\varphi,\tilde g_1,\tilde g_2)$-isomorphism $\psi \colon N_1 \mathrel{\widetilde\to} N_2$ for some elements $\tilde g_1, \tilde g_2 \in G(\tilde H)$
and an algebra isomorphism $\varphi \colon \tilde H \mathrel{\widetilde\to} \tilde H$. 
Denote by $\bar a$ the image of $a\in A$ in $A/J^{\tilde H}(A)$.
Then for any $\bar a, \bar b\in N_2$ we
have $$\bar a \bar b = \bar a b=\psi(\psi^{-1}(\bar a)) b
=\psi(\psi^{-1}(\bar a)(\tilde g_2^{-1} b))
= \psi(\psi^{-1}(\bar a)(\tilde g_2^{-1} \bar b))
= 0$$ since $\psi^{-1}(\bar a)\in N_1$ and $\tilde g_2^{-1} \bar b \in N_2$.
Hence $N_2^2=0$ and we get a contradiction to the fact that $N_2$ is $\tilde H$-simple. Therefore $s=1$
and $A/J^{\tilde H}(A)$
is an $\tilde H$-simple algebra. In particular, $J^{\tilde H}(A)$ is a maximal $\tilde H$-invariant ideal.

We claim now that $J^{\tilde H}(A)$ is the unique maximal $\tilde H$-invariant ideal. Indeed, if $I \subseteq A$ is another $\tilde H$-invariant ideal
such that $I \subsetneqq J^{\tilde H}(A)$, we get $A=I+J^{\tilde H}(A)$
and $$A/(I\cap J^{\tilde H}(A)) \cong I/(I\cap J^{\tilde H}(A)) \oplus
J^{\tilde H}(A) /(I\cap J^{\tilde H}(A))$$ (direct sum of $\tilde H$-invariant ideals). Since by Lemma~\ref{LemmaHSemiSimpleIsUnital} the algebra $A$ is unital, $A/(I\cap J^{\tilde H}(A))$ and, therefore,
$J^{\tilde H}(A)/(I\cap J^{\tilde H}(A))$ must be unital too which contradicts the nilpotency of 
$J^{\tilde H}(A)$. Thus  $J^{\tilde H}(A)$ is indeed the unique maximal $\tilde H$-invariant ideal.

 Since  $A/I_{m-1}$ is $(\varphi^{1-m},g_1^{1-m},g_2^{1-m})$-isomorphic to~$\tilde J$
 and $\tilde J$ is an irreducible $(\tilde H,(A,A))$-bimodule, $A/I_{m-1}$ is an $\tilde H$-simple algebra.
 Thus
$I_{m-1}$ is a maximal $\tilde H$-invariant ideal too and $I_{m-1}=J^{\tilde H}(A)$.
\end{proof}

By Theorem~\ref{TheoremHSimpleOreExtStructure}, we can define an $F$-linear map $\psi 
\colon A \to A$ by the formulas $\psi(v^k a) := v^{k-1} a$
for $1\leqslant k \leqslant m-1$ and $a\in \tilde J$ and $\psi(\tilde J) :=0$.
Note that $\psi^m=0$ and $\psi(A)= J^{\tilde H}(A)$.

Now define $F$-linear maps $\psi_k \colon A/J^{\tilde H}(A)\to A$, $0\leqslant k \leqslant m-1$, as follows.
For every $\bar a \in \bar A := A/J^{\tilde H}(A)$
there exists unique $a\in v^{m-1} \tilde J$ such that $\bar a = a + J^{\tilde H}(A)$. Define 
$\psi_k(\bar a) := \psi^k(a)$, $0\leqslant k \leqslant m-1$.
Then we have $A=\bigoplus_{k=0}^{m-1}\psi_k(\bar A)$. In other words,
the algebra $A$ is built up of images of the algebra $\bar A$.
We will calculate products of $\psi_k(a)$ in the next section.
In the proposition below we study the properties of $\psi_{m-1}$ with respect to the $\tilde H$-action.

\begin{proposition}\label{PropositionHOreActionOnm1} Suppose we are under the conditions of Theorem~\ref{TheoremHSimpleOreExtStructure}.
Then
$h \psi_{m-1}(\bar a)= \psi_{m-1}(\varphi^{m-1}(h)\bar a)$ for all $\bar a\in \bar A$ and $\tilde h \in \tilde H$. 
\end{proposition}
\begin{proof} Since $\psi_{m-1}(\bar a) \in \tilde J$ and $\tilde J$ is an $\tilde H$-invariant ideal,
$h \psi_{m-1}(\bar a)= b $ for some $b\in \tilde J$.
Let $\delta(h):=vh-\varphi(h)v$, $h\in\tilde H$.
 Then 
\begin{equation*}\begin{split}v^{m-1} b = v^{m-1} h \psi_{m-1}(\bar a)=v^{m-2} \varphi(h) v \psi_{m-1}(\bar a)+v^{m-2} \delta(h)
\psi_{m-1}(\bar a) =\\ = \varphi^{m-1}(h) v^{m-1} \psi_{m-1}(\bar a)+ \sum_{i=0}^{m-2}  v^{m-i-2} \delta\left(\varphi^{i}(h)\right) v^i \psi_{m-1}(\bar a)=\\=
\varphi^{m-1}(h) a + \sum_{i=0}^{m-2}  v^{m-i-2} \delta\left(\varphi^{i}(h)\right) v^i \psi_{m-1}(\bar a).\end{split}\end{equation*}
By Theorem~\ref{TheoremHSimpleOreExtStructure}, the second term belongs to $J^{\tilde H}(A)$.
Hence $\pi(v^{m-1} b)= \varphi^{m-1}(h) \bar a$ where $\pi \colon A \twoheadrightarrow A/J^{\tilde H}(A)$
is the natural surjective homomorphism. Since $v^{m-1}b \in v^{m-1} \tilde J$, we have $\psi_0(\varphi^{m-1}(h) \bar a)=v^{m-1}b$ and $h \psi_{m-1}(\bar a)=b= \psi_{m-1}(\varphi^{m-1}(h)\bar a)$.
\end{proof}

\section{Multiplication in $H$-simple algebras}\label{SectionHSimpleOreExtMul}

In this section we show that, under additional assumptions, the multiplication in an $H$-simple
algebra $A$ is induced by that one in the $\tilde H$-simple algebra $A/J^{\tilde H}(A)$. 

In Theorem~\ref{TheoremHSimpleOreExtMul} below we use \textit{quantum binomial coefficients}: $$\binom{n}{k}_\zeta := \frac{n!_\zeta}{(n-k)!_\zeta\ k!_\zeta}$$ where $n!_\zeta := n_\zeta (n-1)_\zeta \cdot \dots \cdot 1_\zeta$ and
$n_\zeta := 1 + \zeta + \zeta^2 + \dots + \zeta^{n-1}$, $n\in\mathbb N$, $0_\zeta :=1$.

Let $H$ be a Hopf algebra over a field $F$ generated as an algebra by a subHopfalgebra $\tilde H$ and a skew-primitive element $v\in H$, $\Delta v = g \otimes v + v \otimes 1$, $g \in G(\tilde H)$.
Also suppose that there exists an algebra automorphism $\varphi \colon \tilde H \to \tilde H$
such that $vh-\varphi(h)v\in \tilde H$ for all $h\in\tilde H$,
$\varphi(g)=\zeta g$ for a primitive $t$th root of unity $\zeta$,
and $v^t \tilde J \subseteq \tilde J$,
for example, $v^t \in \tilde H$.
Let $A$ be a finite dimensional $H$-simple $H$-module algebra with $J^{\tilde H}(A)\ne 0$. 
Let $\tilde J \subseteq J^{\tilde H}(A)^{p-1}$ be a minimal
two-sided $\tilde H$-invariant ideal where $p$ is defined by $J^{\tilde H}(A)^p=0$ and $J^{\tilde H}(A)^{p-1}\ne 0$.
Denote $v^{-1} \tilde J:= 
\lbrace a\in A \mid va\in  \tilde J\rbrace$.

\begin{theorem}\label{TheoremHSimpleOreExtMul}
The number $t$ equals the number $m$ defined in Theorem~\ref{TheoremHSimpleOreExtStructure}, the subspace $B := v^{-1} \tilde J$ is a $g$-invariant subalgebra that coincides with $v^{m-1} \tilde J$ and we have a decomposition  $A=B\oplus J^{\tilde H}(A)$
(direct sum of subspaces), the map $\psi_0 \colon A / J^{\tilde H}(A) \mathrel{\widetilde\to}
B$ is an isomorphism of algebras,
and for every $0\leqslant k,\ell < m$ and $\bar a, \bar b \in A / J^{\tilde H}(A)$
we have \begin{equation}\label{EqQuantumCoefMulHSimpleOreExt}
\psi_k(\bar a)\psi_\ell(\bar b)=\tbinom{k+\ell}{k}_\zeta \psi_{k+\ell}
((g^\ell \bar a) \bar b).
\end{equation} (Here $\psi_k$ are the maps defined at the end of Section~\ref{SectionHSimpleOreExtStructure}.)
\end{theorem}

In order to prove Theorem~\ref{TheoremHSimpleOreExtMul}, we need several lemmas.

Note that $\psi(ga)=\zeta^{-1}g\psi(a)$ for all $a\in A$ where $\psi$ is the map defined in the previous section.

\begin{lemma}\label{LemmaHSimpleOreExtMul0}
We have $\psi(ab)=\psi(a)b$ and $\psi(ba)=(g^{-1}b)\psi(a)$
for all $a \in A$ and $b \in v^{-1} \tilde J$.
\end{lemma}
\begin{proof} 
If $a\in \tilde J$, then the assertion of the lemma is obvious.
Since by Theorem~\ref{TheoremHSimpleOreExtStructure}
we have $A=vJ^{\tilde H}(A) + \tilde J$, it is sufficient to prove the equalities above in
the case when $a = vu$ for some $u\in J^{\tilde H}(A)$.
Note that by the definition $\psi(vu)=u$ for all $u\in J^{\tilde H}(A)$.
Moreover
$J^{\tilde H}(A)$ is an $\tilde H$-invariant two-sided ideal and
$(gu)(vb)\in J^{\tilde H}(A) \tilde J = 0$.
Hence $$\psi(ab)=\psi((vu)b)=\psi(v(ub))- \psi((gu)(vb)) = ub=\psi(vu)b=\psi(a)b.$$
Analogously, \begin{equation*}\begin{split}\psi(ba)=\psi(b(vu))=\psi(v((g^{-1}b)u))- \psi((vg^{-1}b)u) = \\ = (g^{-1}b)u-\zeta^{-1}\psi((g^{-1}vb)u)=(g^{-1}b)\psi(vu)-0=(g^{-1}b)\psi(a).\end{split}\end{equation*}
\end{proof}

Now we can prove a formula involving the powers of $\psi$. 

\begin{lemma}\label{LemmaHSimpleOreExtMul}
Let $a,b \in v^{-1} \tilde J$, $0\leqslant k,\ell < t$.
Then \begin{equation}\label{EqQuantumCoefMulHSimpleOreExt0}
\psi^k(a)\psi^\ell(b)=\tbinom{k+\ell}{k}_\zeta \psi^{k+\ell}
((g^\ell a) b).
\end{equation}
\end{lemma}
\begin{proof}
We prove the theorem by induction on $k+\ell$. If $k=0$ or $\ell = 0$,
(\ref{EqQuantumCoefMulHSimpleOreExt0}) is a consequence
of Lemma~\ref{LemmaHSimpleOreExtMul0}.
Suppose both $k,\ell \geqslant 1$.
Then $\psi^k(a), \psi^\ell(b) \in J^{\tilde H}(A)$.
For every $u\in A$ we have $v\psi(u)-u \in \tilde J$.
Since $ J^{\tilde H}(A)\tilde J = \tilde J J^{\tilde H}(A) = 0$,
we get $(g\psi^k(a))v\psi^\ell(b) = (g \psi^k(a))\psi^{\ell-1}(b)$
and $(v\psi^k(a))\psi^\ell(b) = \psi^{k-1}(a)\psi^\ell(b)$.
Therefore
 \begin{equation*}\begin{split}\psi^k(a)\psi^\ell(b)
= \psi(v(\psi^k(a)\psi^\ell(b)))=\\=\psi((g\psi^k(a))v\psi^\ell(b)
+(v\psi^k(a))\psi^\ell(b))=\psi(\zeta^k(\psi^k(g a))\psi^{\ell-1}(b)
+ \psi^{k-1}(a)\psi^\ell(b))=\\=\psi\left(
 \zeta^k \tbinom{k+\ell-1}k_\zeta \psi^{k+\ell-1}
((g^\ell a) b)
+ \tbinom{k+\ell-1}{k-1}_\zeta \psi^{k+\ell-1}
((g^\ell a) b)\right)=\tbinom{k+\ell}{k}_\zeta \psi^{k+\ell}
((g^\ell a) b).\end{split}\end{equation*}
\end{proof}
\begin{corollary}Let $m$ be the number defined in Theorem~\ref{TheoremHSimpleOreExtStructure}. Then $t=m$ and
$\zeta$ is a primitive $m$th root of unity.
\end{corollary}
\begin{proof} First we notice that the condition
$v^t \tilde J \subseteq \tilde J$ implies
that $I_t$ is an $H$-invariant ideal of $A$ and $m \leqslant t$.
In addition, $m \geqslant 2$ since $I_{m-1}=J^{\tilde H}(A) \ne 0$.
Note that $1_A \notin I_{m-1}$,
since $I_{m-1}$ is a nontrivial ideal. Hence $\psi^{m-1}(1_A) \in J^{\tilde H}(A) \backslash \lbrace 0 \rbrace$.
Since $v\psi(a)-a \in \tilde J$ for all $a\in A$,
we have $v\psi^{m-1}(1_A)=\psi^{m-2}(1_A)+j_1$
and $v\psi(1_A)=1_A+j_2$ for some $j_1, j_2 \in \tilde J$.
Note that $\psi^{m-1}(1_A) \psi(1_A)=\binom{m}{m-1}_\zeta \psi^m(1_A) = 0$.
However \begin{equation*}\begin{split}0=v(\psi^{m-1} (1_A)\psi(1_A))
=(v\psi^{m-1}(1_A))\psi(1_A)+(g\psi^{m-1}(1_A))v\psi(1_A)
=\\
(\psi^{m-2}(1_A)+j_1)\psi(1_A)+\zeta^{m-1}\psi^{m-1}(1_A)(1_A+j_2)
=\psi^{m-2}(1_A)\psi(1_A)+\zeta^{m-1}\psi^{m-2}(1_A)1_A
=\\ \left(\binom {m-1}{m-2}_\zeta+\zeta^{m-1}\right)\psi^{m-1}(1_A)
=m_\zeta\ \psi^{m-2}(1_A)\end{split}\end{equation*}
since $ J^{\tilde H}(A)\tilde J = \tilde J J^{\tilde H}(A) = 0$.
Hence $m_\zeta = 0$ and $t=m$.
\end{proof}
\begin{proof}[Proof of Theorem~\ref{TheoremHSimpleOreExtMul}.] 
By the assumption, $v^{m-1} \tilde J \subseteq v^{-1} \tilde J$.
Note that $$vJ^{\tilde H}(A)=v I_{m-1}=\bigoplus_{i=1}^{m-1} v^i\tilde J$$
and $v^{-1} \tilde J \cap J^{\tilde H}(A)=0$.
Hence $v^{-1} \tilde J = v^{m-1} \tilde J$.
Since $\tilde J$ is an $\tilde H$-invariant ideal,
$$v(ab)=(ga)(vb)+(va)b \in \tilde J\text{ for all }a,b \in v^{-1} \tilde J,$$
and
 $v^{-1} \tilde J$ is a subalgebra. Now $$A=v^{m-1} \tilde J \oplus J^{\tilde H}(A)\qquad\text{(direct sum of subspaces)}$$
implies $A=B\oplus J^{\tilde H}(A)$ and~(\ref{EqQuantumCoefMulHSimpleOreExt0}) implies~(\ref{EqQuantumCoefMulHSimpleOreExt}).
\end{proof}

\section{Polynomial $H$-identities}\label{SectionHPI}

The rest of the paper is devoted to polynomial $H$-identities. 

Denote by $F \langle X \rangle$ the free associative algebra without $1$
   on the set $X := \lbrace x_1, x_2, x_3, \ldots \rbrace$.
  Then $F \langle X \rangle = \bigoplus_{n=1}^\infty F \langle X \rangle^{(n)}$
  where $F \langle X \rangle^{(n)}$ is the linear span of all monomials of total degree $n$.
   Let $H$ be an arbitrary associative algebra with $1$ over $F$. Consider the algebra $$F \langle X | H\rangle
   :=  \bigoplus_{n=1}^\infty H^{{}\otimes n} \otimes F \langle X \rangle^{(n)}$$
   with the multiplication $(u_1 \otimes w_1)(u_2 \otimes w_2):=(u_1 \otimes u_2) \otimes w_1w_2$
   for all $u_1 \in  H^{{}\otimes j}$, $u_2 \in  H^{{}\otimes k}$,
   $w_1 \in F \langle X \rangle^{(j)}$, $w_2 \in F \langle X \rangle^{(k)}$.
We use the notation $$x^{h_1}_{i_1}
x^{h_2}_{i_2}\ldots x^{h_n}_{i_n} := (h_1 \otimes h_2 \otimes \ldots \otimes h_n) \otimes x_{i_1}
x_{i_2}\ldots x_{i_n}.$$ Here $h_1 \otimes h_2 \otimes \ldots \otimes h_n \in H^{{}\otimes n}$,
$x_{i_1} x_{i_2}\ldots x_{i_n} \in F \langle X \rangle^{(n)}$. 
In addition, we identify $x_i$ and $x_i^{1_H}$ and treat $X$ as a subset of 
$F \langle X | H\rangle$.

If $(\gamma_\beta)_{\beta \in \Lambda}$ is a basis in $H$, 
then $F\langle X | H \rangle$ is isomorphic to the free non-unital associative algebra over $F$ with free formal  generators $x_i^{\gamma_\beta}$, $\beta \in \Lambda$, $i \in \mathbb N$.
 We refer to the elements
 of $F\langle X | H \rangle$ as \textit{$H$-polynomials}.
 
Let $A$ be an associative algebra with a generalized $H$-action.
Any map $\psi \colon X \to A$ has the unique homomorphic extension $\bar\psi
\colon F \langle X | H \rangle \to A$ such that $\bar\psi(x_i^h)=h\psi(x_i)$
for all $i \in \mathbb N$ and $h \in H$.
 An $H$-polynomial
 $f \in F\langle X | H \rangle$
 is a \textit{polynomial $H$-identity} of $A$ if $\bar\psi(f)=0$
for all maps $\psi \colon X \to A$. In other words, $f(x_1, x_2, \ldots, x_n)$
 is an $H$-identity of $A$
if and only if $f(a_1, a_2, \ldots, a_n)=0$ for any $a_i \in A$.
 In this case we write $f \equiv 0$.
The set $\Id^{H}(A)$ of all polynomial $H$-identities
of $A$ is an ideal of $F\langle X | H \rangle$.

\begin{remark}
Note that here we do not consider any $H$-action on $F \langle X | H \rangle$.
However, extending the category of algebras with a generalized $H$-action
in a proper way, we can make $F \langle - | H \rangle$ a free functor corresponding to a free-forgetful adjunction. (In~\cite[Section 7.2]{ASGordienko16} the author considers the non-associative case, but similar observations
can be made in the case of associative and Lie algebras too.)
\end{remark}

Denote by $P^H_n$ the space of all multilinear $H$-polynomials
in $x_1, \ldots, x_n$, $n\in\mathbb N$, i.e.
$$P^{H}_n = \langle x^{h_1}_{\sigma(1)}
x^{h_2}_{\sigma(2)}\ldots x^{h_n}_{\sigma(n)}
\mid h_i \in H, \sigma\in S_n \rangle_F \subset F \langle X | H \rangle.$$
The number $c^H_n(A):=\dim\left(\frac{P^H_n}{P^H_n \cap \Id^H(A)}\right)$
is called the $n$th \textit{codimension of polynomial $H$-identities}
or the $n$th \textit{$H$-codimension} of $A$.

\begin{remark} Every algebra $A$ is an $H$-module algebra
for $H=F$. In this case the $H$-action is trivial and we get ordinary polynomial identities and their codimensions $c_n(A)$. 
\end{remark}
\begin{remark}
If $A$ is a finite dimensional algebra graded by a set $T$, then one can introduce the notion of graded polynomial identities and their codimensions $c_n^{T\text{-}\mathrm{gr}}(A)$. However, by~\cite[Lemma~4.3, Remarks~3.3 and~4.4]{ASGordienko16} we have $c_n^{T\text{-}\mathrm{gr}}(A)=c_n^{F^T}(A)$ for all $n\in \mathbb N$.
\end{remark}

Now we give an example of an infinite dimensional $H$-module algebra for an infinite dimensional Hopf
algebra $H$ such that all $H$-codimensions are infinite:

\begin{example}\label{ExampleInfCodim} Let be $F$ be a field and let $(G, \cdot)$ be the group $(\mathbb Q, +)$ written in the multiplicative form. Fix a number $m \in\mathbb N$
where $m \geqslant 2$. Denote by $\varphi$ the automorphism $G \to G$ where $\varphi(g)=g^m$
for all $g\in G$. 
Let $A: = FG$ be the group algebra of $G$. Then $\varphi$ is naturally extended to
 an automorphism $FG \mathrel{\widetilde\to} FG$, and $A$ becomes an algebra with the following $\mathbb Z$-action:
$b^k := \underbrace{\varphi(\varphi(\ldots (\varphi}_k(b)\ldots)$ for $k\in\mathbb N$
and $b\in A$. 
We claim that despite the fact that the algebra $A$ is commutative, 
$c_n^{F\mathbb Z}(A) = +\infty$ for all $n\in\mathbb N$.
Indeed, if $\omega \in G$ is the element corresponding to $2\in\mathbb Q$,
then the substitution $x_1=x_2=\dots=x_n=\omega$ shows that
the multilinear $F\mathbb Z$-polynomials $x_1^k x_2 \cdots x_n$, $k\in\mathbb Z$, are linearly independent modulo $\Id^{F\mathbb Z}(A)$.
(Recall that here $x_1^k$ denotes the $\mathbb Z$-action and not the raising to the $k$th power.)
\end{example}

One of the main tools in the investigation of polynomial
identities is provided by the representation theory of symmetric groups.
 The symmetric group $S_n$  acts
 on the space $\frac {P^H_n}{P^H_{n}
  \cap \Id^H(A)}$
  by permuting the variables. If the base field $F$ is of characteristic $0$,
  then irreducible $FS_n$-modules are described by partitions
  $\lambda=(\lambda_1, \ldots, \lambda_s)\vdash n$ and their
  Young diagrams $D_\lambda$.
   The character $\chi^H_n(A)$ of the
  $FS_n$-module $\frac {P^H_n}{P^H_n
   \cap \Id^H(A)}$ is
   called the $n$th
  \textit{cocharacter} of polynomial $H$-identities of $A$.
  We can rewrite it as
  a sum $$\chi^H_n(A)=\sum_{\lambda \vdash n}
   m(A, H, \lambda)\chi(\lambda)$$ of
  irreducible characters $\chi(\lambda)$.
%  The number $\ell_n^H(A):=\sum_{\lambda \vdash n}
%   m(A, H, \lambda)$ is called the $n$th
%  \textit{colength} of polynomial $H$-identities of $A$.
Let  $e_{T_{\lambda}}=a_{T_{\lambda}} b_{T_{\lambda}}$
and
$e^{*}_{T_{\lambda}}=b_{T_{\lambda}} a_{T_{\lambda}}$
where
$a_{T_{\lambda}} = \sum_{\pi \in R_{T_\lambda}} \pi$
and
$b_{T_{\lambda}} = \sum_{\sigma \in C_{T_\lambda}}
 (\sign \sigma) \sigma$,
be the Young symmetrizers corresponding to a Young tableau~$T_\lambda$.
Then $M(\lambda) = FS_n e_{T_\lambda} \cong FS_n e^{*}_{T_\lambda}$
is an irreducible $FS_n$-module corresponding to
 a partition~$\lambda \vdash n$.
  We refer the reader to~\cite{Bahturin, DrenKurs, ZaiGia}
   for an account
  of $S_n$-representations and their applications to polynomial
  identities.

\section{Property (*)} \label{SectionPropertyStar}

In order to formulate our assumptions on $H$-simple algebras,
 we introduce the following property: 
\begin{enumerate}
\item[(*)] Suppose $B$ is a finite dimensional algebra with a generalized $H$-action for some associative algebra $H$ with $1$ over a field $F$. Let $a_1, \ldots, a_\ell$ be a basis of $B$.
 We say that $B$ satisfies Property (*) if $B$ is unital and
there exists a number $n_0 \in \mathbb N$ such that for every $k \in\mathbb N$
there exist a multilinear $H$-polynomial $$f=f(x_1^{(1)}, \ldots, x_\ell^{(1)}; \ldots;
x^{(2k)}_1, \ldots,  x^{(2k)}_\ell;\ z_1, \ldots, z_{n_1})$$
and elements $\bar z_i \in B$, $1\leqslant i \leqslant n_1$, $0\leqslant n_1 \leqslant n_0$, such that $f$ is alternating in $x_1^{(i)}, \ldots, x_\ell^{(i)}$
for each $1\leqslant i  \leqslant 2k$
and $f(a_1, \ldots, a_\ell; \ldots;
a_1, \ldots, a_\ell; \bar z_1, \ldots, \bar z_{n_1}) \ne 0$.
\end{enumerate}

\begin{example}\label{ExampleHSemiSimplePropertyStar}
If $B$ is a finite dimensional semisimple $H$-simple algebra with a generalized $H$-action over an algebraically closed field of characteristic $0$ (e.g. $B$ is an $H$-simple $H$-module algebra for a finite dimensional semisimple Hopf algebra $H$, see~\cite[Theorem~3.8]{LinMontSmall}), then by~\cite[Theorem~7]{ASGordienko3} the algebra $B$ satisfies Property (*). (In~\cite[Theorem~7]{ASGordienko3} the author requires $\dim H < +\infty$, but one can replace
$H$ with its image $H_1$ in $\End_F(A)$ and notice that $\dim H_1 < +\infty$ and $c_n^H(A) = c_n^{H_1}(A)$ for all $n\in\mathbb N$.) 
\end{example}

\begin{example}
If $B$ is a finite dimensional $H_{m^2}(\zeta)$-simple $H_{m^2}(\zeta)$-module algebra
over an algebraically closed field of characteristic $0$ for a Taft algebra $H_{m^2}(\zeta)$, then 
by~\cite[Lemma~7]{ASGordienko12} the algebra $B$ satisfies Property (*).
\end{example}

In fact, a more general theorem holds:

\begin{theorem}\label{TheoremHOrePropertyStar}
Let $H$ be a Hopf algebra over an algebraically closed field $F$ of characteristic $0$ generated as an algebra by a subHopfalgebra $\tilde H$ and a skew-primitive element $v\in H$, $\Delta v = g_1 \otimes v + v \otimes g_2$, $g_1, g_2 \in G(\tilde H)$.
Also suppose that there exists an algebra automorphism $\varphi \colon \tilde H \to \tilde H$
such that $vh-\varphi(h)v\in \tilde H$ for all $h\in\tilde H$.
Suppose that all finite dimensional $\tilde H$-simple algebras
satisfy Property (*). Then all finite dimensional $H$-simple algebras
satisfy Property (*) too.
\end{theorem}
\begin{corollary}\label{CorollaryOreExtStar}
If $B$ is a finite dimensional $H$-simple $H$-module algebra
for a Hopf algebra $H$ over an algebraically closed field of characteristic $0$ such that $H$ is constructed by an iterated Ore extension of a finite dimensional semisimple Hopf algebra by skew-primitive elements, then the algebra $B$ satisfies Property (*).
\end{corollary}

\begin{proof}[Proof of Theorem~\ref{TheoremHOrePropertyStar}.]
Let $A$ be a finite-dimensional $H$-simple algebra.
If $A$ is $\tilde H$-simple, then $A$ satisfies Property (*).
Now Lemma~\ref{LemmaHOreRadical} implies that without loss of generality we may assume that $J^{\tilde H}(A) \ne 0$.
By Theorem~\ref{TheoremHSimpleOreExtStructure},
there exists $m\in\mathbb N$ such that
$A = \bigoplus_{k=0}^{m-1}v^k \tilde J$, $J^{\tilde H}(A)=
\bigoplus_{k=0}^{m-2}v^k \tilde J$.
Moreover, $A/J^{\tilde H}(A)$ is a $\tilde H$-simple algebra.
By the assumptions of Theorem~\ref{TheoremHOrePropertyStar},
$A/J^{\tilde H}(A)$ satisfies Property (*).
Let $a_1, \ldots, a_\ell$ be a basis of $\tilde J$.
Then the images $b_1, \ldots, b_\ell$ of $v^{m-1}a_1, \ldots, v^{m-1}a_\ell$
in $A/J^{\tilde H}(A)$ form a basis of $A/J^{\tilde H}(A)$.

By Property (*) there exists 
$n_0 \in \mathbb N$ such that for every $k \in\mathbb N$
there exist a multilinear $\tilde H$-polynomial $$\tilde f= \tilde f(x_1^{(1)}, \ldots, x_\ell^{(1)}; \ldots;
x^{(2k)}_1, \ldots,  x^{(2k)}_\ell;\ z_1, \ldots, z_{n_1})$$
and elements $\bar z_i \in A/J^{\tilde H}(A)$, $1\leqslant i \leqslant n_1$, $0\leqslant n_1 \leqslant n_0$, such that $\tilde f$ is alternating in $x_1^{(i)}, \ldots, x_\ell^{(i)}$
for each $1\leqslant i  \leqslant 2k$
and $b:=\tilde f(b_1, \ldots, b_\ell; \ldots;
b_1, \ldots, b_\ell; \bar z_1, \ldots, \bar z_{n_1}) \ne 0$.
Since the $\tilde H$-invariant ideal generated by $b$ coincides
with $A/J^{\tilde H}(A)$ and $\left(A/J^{\tilde H}(A)\right)^2=A/J^{\tilde H}(A)$, there exist $h_1, \ldots, h_m
\in\tilde H$ and $\bar w_1, \ldots, \bar w_{m-1} \in A/J^{\tilde H}(A)$ such that $$(h_1 b) \bar w_1 (h_2 b) \bar w_2 \ldots (h_{m-1} b) \bar w_{m-1} (h_m b) \ne 0.$$
Define the following multilinear functions that can be presented 
as $\tilde H$- and $H$-polynomials, respectively, using~(\ref{EqGeneralizedHopf}):
 \begin{equation*}\begin{split}
\tilde f_0 := \left(\tilde f(x_{11}^{(1)}, \ldots, x_{1\ell}^{(1)}; \ldots;
x^{(2k)}_{11}, \ldots,  x^{(2k)}_{1\ell};\ z_{11}, \ldots, z_{1n_1})\right)^{h_1}\cdot\\ \prod_{i=2}^m w_{i-1} \left(\tilde f(x_{i1}^{(1)}, \ldots, x_{i\ell}^{(1)}; \ldots;
x^{(2k)}_{i1}, \ldots,  x^{(2k)}_{i\ell};\ z_{i1}, \ldots, z_{in_1})\right)^{h_i}
\end{split}\end{equation*}
and 
\begin{equation*}\begin{split}
f_0 := \left(\tilde f(x_{11}^{(1)}, \ldots, x_{1\ell}^{(1)}; \ldots;
x^{(2k)}_{11}, \ldots,  x^{(2k)}_{1\ell};\ z_{11}, \ldots, z_{1n_1})\right)^{h_1}\cdot\\ \prod_{i=2}^m w_{i-1} \left(\tilde f\left(\left(x_{i1}^{(1)}\right)^{v^{i-1}}, \ldots, \left( x_{i\ell}^{(1)}\right)^{v^{i-1}}; \ldots;
\left(x^{(2k)}_{i1} \right)^{v^{i-1}}, \ldots,  \left( x^{(2k)}_{i\ell} \right)^{v^{i-1}};\ z_{i1}, \ldots, z_{in_1}\right)\right)^{h_i}.
\end{split}\end{equation*}

Then $\tilde f_0$ does not vanish under the substitution
$x_{ij}^{(t)}=b_j$ for $1\leqslant j \leqslant \ell$, $1\leqslant t \leqslant 2k$,
$z_{ij}=\bar z_j$ for $1\leqslant j \leqslant n_1$,
$1\leqslant i \leqslant m$, and
$w_i = \bar w_i$ for $1\leqslant i \leqslant m-1$.
Considering some preimages $u_i$ of $\bar w_i$
and some preimages $q_i$ of $\bar z_i$ in $A$,
 we obtain that the value of $f_0$ under the substitution
$x_{ij}^{(t)}=v^{m-i}a_j$ for $1\leqslant j \leqslant \ell$, $1\leqslant t \leqslant 2k$,
$z_{ij}=q_j$ for $1\leqslant j \leqslant n_1$,
$1\leqslant i \leqslant m$, and
$w_i = u_i$ for $1\leqslant i \leqslant m-1$,
does not belong to $J^{\tilde H}(A)$.

Denote the last substitution above by $\Xi$ and the value of $f_0$ under the substitution $\Xi$ by $a$.

Let $f := \Alt_1 \ldots \Alt_{2k} f_0$ where $\Alt_t$
is the operator of alternation in $x^{(t)}_{ij}$, 
$1\leqslant i \leqslant m$, $1\leqslant j \leqslant \ell$.
Consider the image of the value of $f$ under the substitution $\Xi$
in $A/J^{\tilde H}(A)$. If an alternation replaces $x^{(t)}_{ij}$
with $x^{(t)}_{i'j'}$ where $i < i'$, then the value
of $(x^{(t)}_{i'j'})^{v^{i-1}}$ under $\Xi$ equals $v^{m-1+(i-i')}a_{j'}
\in J^{\tilde H}(A)$, i.e. the image of the whole term in 
$A/J^{\tilde H}(A)$ is zero. Therefore, working modulo $J^{\tilde H}(A)$,
we may assume that the alternations replace $x^{(t)}_{ij}$
with $x^{(t)}_{ij'}$ for the same $i,t$. Taking into account that
$\tilde f$ is alternating in $x_1^{(i)}, \ldots, x_\ell^{(i)}$
for each $1\leqslant i  \leqslant 2k$,
we obtain that the value of $f$ under the substitution $\Xi$
belongs to $(\ell!)^{2km}a+J^{\tilde H}(A)$, i.e. is nonzero.

The polynomial $f$ satisfy all the requirements of Property (*)
for $H$. Therefore, all finite dimensional $H$-simple algebras
satisfy Property (*). 
\end{proof}

\section{Growth of polynomial $H$-identities}\label{SectionGrowthHIdAssoc}

Now we are ready to formulate the main theorem which will be proved at the end of the paper:

\begin{theorem}\label{TheoremHmodHRadAmitsurPIexpHBdimB}
Let $A$ be a finite dimensional non-nilpotent associative algebra with a generalized $H$-action for some associative algebra $H$ with $1$ over an algebraically closed field $F$ of characteristic $0$. Suppose $A/J^H(A) = B_1 \oplus B_2 \oplus \ldots
\oplus B_q$ (direct sum of $H$-invariant ideals) for some $H$-simple algebras $B_i$
satisfying Property~(*).
Let $\varkappa \colon A/J^H(A) \hookrightarrow A$ be the $(B, B)$-bimodule embedding
from Lemma~\ref{LemmaWeakWedderburnMalcevHRad} and 
let $$d:= \max\dim\left( B_{i_1}\oplus B_{i_2} \oplus \ldots \oplus B_{i_r}
 \mathbin{\Bigl|}  r \geqslant 1,\right.$$
  \begin{equation}\label{EqdAssoc} \left. (H\varkappa(B_{i_1}))A^+ \,(H\varkappa(B_{i_2})) A^+ \ldots (H\varkappa(B_{i_{r-1}})) A^+\,(H\varkappa(B_{i_r}))\ne 0\right)
  \end{equation} where
 $A^+:=A+F\cdot 1$.
Then there exist $C_1,C_2 > 0$ and $r_1,r_2 \in \mathbb R$ such that $$C_1 n^{r_1} d^n \leqslant c^H_n(A) \leqslant C_2 n^{r_2} d^n$$
for all $n\in\mathbb N$.
\end{theorem}

\begin{remark}
If $A$ is nilpotent, i.e. $x_1 \ldots x_p\equiv 0$ for some $p\in\mathbb N$, then
$P^{H}_n \subseteq \Id^{H}(A)$ and $c^H_n(A)=0$ for all $n \geqslant p$.
\end{remark}
\begin{corollary}
Consequently, there exists integer $\PIexp^H(A):=\lim_{n\to\infty}\sqrt[n]{c_n^H(A)}= d$ and the analog of Amitsur's conjecture holds for $A$.
\end{corollary}

\begin{corollary}\label{CorollaryHOreAmitsur}
Let $H$ be a Hopf algebra over a field $F$ of characteristic $0$ such that $H$ is constructed by an iterated Ore extension of a finite dimensional semisimple Hopf algebra by skew-primitive elements (e.g. $H$ is a Taft algebra $H_{m^2}(\zeta)$). Let $A$ be a finite dimensional non-nilpotent associative $H$-module algebra.
 Then there exist constants $C_1, C_2 > 0$, $r_1, r_2 \in \mathbb R$,
  $d \in \mathbb N$ such that $C_1 n^{r_1} d^n \leqslant c^{H}_n(A)
   \leqslant C_2 n^{r_2} d^n$ for all $n \in \mathbb N$.
   In particular, the analog of Amitsur's conjecture holds
 for $A$.
\end{corollary}
\begin{proof}
Note that $H$-codimensions do not change upon an extension of the base field.
The proof is analogous to the case of ordinary codimensions~\cite[Theorem~4.1.9]{ZaiGia}.
Hence we may assume the base field $F$ to be algebraically closed.
By Corollary~\ref{CorollaryOreExtStar} all finite dimensional $H_{m^2}(\zeta)$-simple algebras
over $F$ satisfy Property (*).
By Theorem~\ref{TheoremSkryabinVanOystaeyen} and Lemma~\ref{LemmaHSemiSimpleIsUnital}, $A/J^{H_{m^2}(\zeta)}(A) = B_1 \oplus \ldots \oplus B_q$ (direct sum of $H_{m^2}(\zeta)$-invariant ideals) for some $H_{m^2}(\zeta)$-simple algebras $B_i$. Now we apply Theorem~\ref{TheoremHmodHRadAmitsurPIexpHBdimB}.
\end{proof}

  In~\cite[Theorem 2]{ASGordienko8} the author required invertibility of the antipode.
Using Theorem~\ref{TheoremSkryabinVanOystaeyen}, one can remove this assumption:

\begin{theorem}\label{TheoremHmoduleAssoc}
Let $A$ be a finite dimensional non-nilpotent associative $H$-module algebra for a Hopf algebra $H$
over a field $F$ of characteristic $0$. Suppose that the Jacobson radical $J(A)$ is an $H$-submodule.
Then there exist constants $d\in\mathbb N$, $C_1, C_2 > 0$, $r_1, r_2 \in \mathbb R$ such that $$C_1 n^{r_1} d^n \leqslant c^{H}_n(A) \leqslant C_2 n^{r_2} d^n\text{ for all }n \in \mathbb N.$$
\end{theorem}
\begin{proof}
Let $K \supset F$ be an extension of the field $F$.
Since $A/J(A)$ is semisimple and $\ch F= 0$, $$(A \otimes_F K)/(J(A) \otimes_F K) \cong (A/J(A)) \otimes_F K$$
is again a semisimple algebra (see e.g.~\cite[Section~10.7, Corollary b]{PierceAssoc}). Since $J(A) \otimes_F K$ is nilpotent, $J(A \otimes_F K) = J(A) \otimes_F K$.
In particular, $J(A \otimes_F K)$ is still $H\otimes_F K$-invariant.
Again, $H$-codimensions do not change upon an extension of the base field. 
Thus we may assume $F$ to be algebraically closed.
By Theorem~\ref{TheoremSkryabinVanOystaeyen}, $A/J(A) = B_1 \oplus \ldots \oplus B_q$ (direct sum of $H$-invariant ideals) for some $H$-simple algebras $B_i$. Since $A/J(A)$ is semisimple, $B_i$ are semisimple too. By~\cite[Theorem~7]{ASGordienko3}, $B_i$ satisfy Property~(*).
Now we apply Theorem~\ref{TheoremHmodHRadAmitsurPIexpHBdimB}.
\end{proof}

Note that the upper bound from Theorem~\ref{TheoremHmodHRadAmitsurPIexpHBdimB} was proved in~\cite[Lemma 2]{ASGordienko8}. In order to prove the lower bound, we construct
a polynomial alternating in sufficiently many sufficiently large sets of variables.

\begin{lemma}\label{LemmaAssocLowerPolynomial}
Let $A$, $\varkappa$, $B_i$, and $d$ be the same as in Theorem~\ref{TheoremHmodHRadAmitsurPIexpHBdimB}.
If $d > 0$, then there exists a number $n_0 \in \mathbb N$ such that for every $n\geqslant n_0$
there exist disjoint subsets $X_1$, \ldots, $X_{2k} \subseteq \lbrace x_1, \ldots, x_n
\rbrace$, $k := \left[\frac{n-n_0}{2d}\right]$,
$|X_1| = \ldots = |X_{2k}|=d$ and a polynomial $f \in P^H_n \backslash
\Id^H(A)$ alternating in the variables of each set $X_j$.
\end{lemma}
\begin{proof} Let $J := J^H(A)$.
Without loss of generality,
we may assume that $$d = \dim(B_1 \oplus B_2 \oplus \ldots \oplus B_r)$$
where 
$(H\varkappa(B_1))A^+ (H\varkappa(B_2))A^+ \ldots (H\varkappa(B_{r-1}))A^+ (H\varkappa(B_r))\ne 0$. 
  
 Since $J$ is nilpotent, we can find maximal $\sum_{i=1}^r q_i$, $q_i \in \mathbb Z_+$,
 such that
$$\left(a_1 \prod_{i=1}^{q_1} j_{1i}\right) ({\gamma_1}\varkappa(b_1))
\left(a_2 \prod_{i=1}^{q_2} j_{2i}\right)
 ({\gamma_2}\varkappa(b_2)) \ldots \left(a_r \prod_{i=1}^{q_r} j_{ri}\right)
 ({\gamma_r}\varkappa(b_r))  \left(a_{r+1} \prod_{i=1}^{q_{r+1}} j_{r+1,i}\right) \ne 0$$ for
  some $j_{ki}\in J$, $a_k \in A^+$, $b_k \in B_i$, $\gamma_k \in H$.
 Let $j_k := a_k\prod_{i=1}^{q_k} j_{ki}$.
 
 Then \begin{equation}\label{EqAssocNonZero}j_1 ({\gamma_1}\varkappa(b_1))
j_2  ({\gamma_2}\varkappa(b_2)) \ldots j_r  ({\gamma_r}\varkappa(b_r))j_{r+1} \ne 0\end{equation}
 for some $b_i \in B_i$, $\gamma_i \in H,$
 however
 \begin{equation}\label{EqAssocbazero}j_1
  \tilde b_1 j_2 \tilde b_2
  \ldots j_r \tilde b_r j_{r+1} = 0 \end{equation} for all $\tilde b_i\in A^+(H\varkappa(B_i))A^+$  such that $\tilde b_k\in J(H\varkappa(B_k))A^+ + A^+(H\varkappa(B_k))J$ for at least one $k$.

Let $a^{(i)}_{k}$, $1 \leqslant k \leqslant d_i := \dim B_i$,
 be a basis in $B_{i}$, $1 \leqslant i \leqslant r$.
 
 By the virtue of Property~(*),
there exist constants $\tilde m_i \in \mathbb Z_+$
such that for any $k$ there exist
 multilinear polynomials $$f_i=f_i(x^{(i, 1)}_1,
 \ldots, x^{(i, 1)}_{d_i};
 \ldots;  x^{(i, 2k)}_1,
 \ldots, x^{(i, 2k)}_{d_i}; z^{(i)}_1, \ldots, z^{(i)}_{m_i}) \in P^H_{2k d_i+m_i} \backslash \Id^H(B_i),$$
   $0 \leqslant m_i \leqslant \tilde m_i$,
alternating in the variables from disjoint sets
$X^{(i)}_{\ell}=\lbrace x^{(i, \ell)}_1, x^{(i, \ell)}_2,
\ldots, x^{(i, \ell)}_{d_i} \rbrace$, $1 \leqslant \ell \leqslant 2k$.
In particular, there exist $\bar z^{(i)}_\alpha \in B_i$, $1 \leqslant \alpha \leqslant m_i$,
such that
$$\hat b_i := f_i(a^{(i)}_1,
 \ldots, a^{(i)}_{d_i};
 \ldots;  a^{(i)}_1,
 \ldots, a^{(i)}_{d_i}; \bar z^{(i)}_1, \ldots, \bar z^{(i)}_{m_i})\ne 0.$$
 
 Let $n_0 := 3r-1+\sum_{i=1}^r \tilde m_i$, $k := \left[\frac{n-n_0}{2d}\right]$, $\tilde k := \left[\frac{n-2kd}{2d_1}\right]+1$. 
Fix polynomials $f_i$, $1 \leqslant i \leqslant r$, for this particular choice of $k$.
 In addition, again by Property~(*),
  take $\tilde f_1=\tilde f_1(x^{(1)}_1,
 \ldots, x^{(1)}_{d_i};
 \ldots;  x^{(2\tilde k)}_1,
 \ldots, x^{(2\tilde k)}_{d_i}; z_1, \ldots, z_{\hat m_1}) \in P^H_{2\tilde k d_1+\hat m_1} \backslash \Id^H(B_1)$ where $0\leqslant \hat m_1 \leqslant \tilde m_1$
 and 
$$\hat b := \tilde f_1(a^{(1)}_1,
 \ldots, a^{(1)}_{d_1};
 \ldots;  a^{(1)}_1,
 \ldots, a^{(1)}_{d_1}; \bar z_1, \ldots, \bar z_{\hat m_1})\ne 0$$
 for some $\bar z_1, \ldots, \bar z_{\hat m_1} \in B_1$.

 Since $B_i$ are $H$-simple, there exist elements
 $h_{i\ell} \in H$, $b_{i\ell}, \tilde b_{i\ell} \in B_i$, $\tilde b_\ell \in B_1$ 
 such that $\sum_\ell b_{i\ell} (h_{i\ell}\hat b_i) \tilde b_{i\ell} = b_i$
 for all $2\leqslant i \leqslant r$ and $\sum_\ell \tilde b_\ell (h_{0\ell}\hat b)  b_{1\ell} (h_{1\ell}\hat b_1) \tilde b_{1\ell} = b_1$.
 
 Now \begin{equation*}\begin{split}
 j_1 \Biggl(\gamma_1\varkappa \biggl(\sum_{s_1} \tilde b_{s_1} \left( h_{0 s_1}
\tilde f_1\bigl( a^{(1)}_1,
 \ldots, a^{(1)}_{d_1};
 \ldots;  a^{(1)}_1,
 \ldots, a^{(1)}_{d_1}; \bar z_1, \ldots, \bar z_{\hat m_1}\bigr)\right)  b_{1 s_1}  \cdot \\ \cdot
   \left( h_{1 s_1} f_1\bigl(a^{(1)}_1,
 \ldots, a^{(1)}_{d_1};
 \ldots;  a^{(1)}_1,
 \ldots, a^{(1)}_{d_1}; \bar z^{(1)}_1, \ldots, \bar z^{(1)}_{m_1} \bigr)\right) \tilde b_{1 s_1} \biggr) 
 \Biggr)j_2
 \cdot \\ \cdot
    \prod_{i=2}^{r}\Biggl(\gamma_i\varkappa\biggl(\sum_{s_i} b_{is_i} \left(h_{i s_i}  f_i\bigl(a^{(i)}_1,
 \ldots, a^{(i)}_{d_i};
 \ldots;  a^{(i)}_1,
 \ldots, a^{(i)}_{d_i}; \bar z^{(i)}_1, \ldots, \bar z^{(i)}_{m_i}\bigr)\right)\tilde b_{is_i} \biggr) \Biggr)j_{i+1}
 \end{split}
 \end{equation*} 
  equals the left-hand side of~(\ref{EqAssocNonZero}), which is nonzero.
  Therefore we can fix indices $s_1, \ldots, s_r$ such that 
\begin{equation*}\begin{split}
 a:= j_1 \Biggl(\gamma_1\varkappa \biggl(\tilde b_{s_1} \left( h_{0 s_1}
\tilde f_1\bigl( a^{(1)}_1,
 \ldots, a^{(1)}_{d_1};
 \ldots;  a^{(1)}_1,
 \ldots, a^{(1)}_{d_1}; \bar z_1, \ldots, \bar z_{\hat m_1}\bigr)\right) b_{1 s_1}  \cdot \\ \cdot
   \left( h_{1 s_1} f_1\bigl(a^{(1)}_1,
 \ldots, a^{(1)}_{d_1};
 \ldots;  a^{(1)}_1,
 \ldots, a^{(1)}_{d_1}; \bar z^{(1)}_1, \ldots, \bar z^{(1)}_{m_1} \bigr)\right) \tilde b_{1 s_1} \biggr)
 \Biggr) 
 j_2 \cdot \\ \cdot
   \prod_{i=2}^{r}\Biggl(\gamma_i\varkappa\biggl(b_{is_i} \left(h_{i s_i}  f_i\bigl(a^{(i)}_1,
 \ldots, a^{(i)}_{d_i};
 \ldots;  a^{(i)}_1,
 \ldots, a^{(i)}_{d_i}; \bar z^{(i)}_1, \ldots, \bar z^{(i)}_{m_i}\bigr)\right)\tilde b_{is_i} \biggr) \Biggr)j_{i+1}\ne 0.
 \end{split}
 \end{equation*}   

Let $B$ be the maximal semisimple subalgebra of $A/J^H(A)$ fixed in Lemma~\ref{LemmaWeakWedderburnMalcevHRad}. Since by Property~(*)
 all $B_i$ are unital, $A/J^H(A)=B_1 \oplus B_2 \oplus \ldots
\oplus B_q$ (direct sum of $H$-invariant ideals) is unital too and $1_B = 1_{A/J^H(A)}$.
Let $\tilde B_i := 1_{B_i} B$. 
Since $\tilde B_i$ are homomorphic images of the semisimple algebra $B$, they are semisimple too.
Now $B \subseteq \tilde B_1 \oplus \tilde B_2 \oplus \ldots
\oplus \tilde B_q$ and the maximality of $B$ imply 
 $B= \tilde B_1 \oplus \tilde B_2 \oplus \ldots
\oplus \tilde B_q$ (direct sum of ideals) and $\tilde B_i$ are maximal semisimple subalgebras
of $B_i$. Hence $1_{\tilde B_i} = 1_{B_i}$ and $1_{B_i} \in B$.
 
 Recall that $\varkappa$ is a homomorphism of $(B,B)$-bimodules.
In particular, $\varkappa(b)=\varkappa(b)\varkappa(1_{B_i})$
for every $b \in B_i$. Hence
\begin{equation*}\begin{split}
 a= j_1 \Biggl(\gamma_1\biggl(\varkappa \Bigl(\tilde b_{s_1} \bigl( h_{0 s_1}
\tilde f_1 ( a^{(1)}_1,
 \ldots, a^{(1)}_{d_1};
 \ldots;  a^{(1)}_1,
 \ldots, a^{(1)}_{d_1}; \bar z_1, \ldots, \bar z_{\hat m_1} ) \bigr) b_{1 s_1}   \cdot \\ \cdot
   \bigl( h_{1 s_1} f_1 (a^{(1)}_1,
 \ldots, a^{(1)}_{d_1};
 \ldots;  a^{(1)}_1,
 \ldots, a^{(1)}_{d_1}; \bar z^{(1)}_1, \ldots, \bar z^{(1)}_{m_1} ) \bigr) \tilde b_{1 s_1} \Bigr)
 \varkappa(1_{B_1}) \biggr) \Biggr) j_2
 \cdot \\ \cdot
    \prod_{i=2}^{r}\Biggl(\gamma_i\biggl(\varkappa\Bigl(b_{is_i} \bigl(h_{i s_i}  f_i (a^{(i)}_1,
 \ldots, a^{(i)}_{d_i};
 \ldots;  a^{(i)}_1,
 \ldots, a^{(i)}_{d_i}; \bar z^{(i)}_1, \ldots, \bar z^{(i)}_{m_i} ) \bigr) \tilde b_{is_i} \Bigr)
  \varkappa(1_{B_i})\biggr)\Biggr) j_{i+1}\ne 0.
 \end{split}
 \end{equation*}   
Moreover $\pi(h\varkappa(a)-\varkappa(ha))=0$ and $\pi(\varkappa(a)\varkappa(b)-\varkappa(ab))=0$ 
imply $h\varkappa(a)-\varkappa(ha) \in J$
and $\varkappa(a)\varkappa(b)-\varkappa(ab) \in J$ for all $a,b\in A$ and $h\in H$.
Hence by~(\ref{EqAssocbazero}) in the entries to the left of $\varkappa(1_{B_i})$
the map $\varkappa$ behaves like a homomorphism of $H$-modules and
\begin{equation*}\begin{split}
 a= j_1 \Biggl(\gamma_1\biggl(\varkappa (\tilde b_{s_1}) \Bigl( h_{0 s_1}
\tilde f_1\bigl( \varkappa (a^{(1)}_1),
 \ldots, \varkappa(a^{(1)}_{d_1});
 \ldots;  \varkappa(a^{(1)}_1),
 \ldots, \varkappa(a^{(1)}_{d_1});
\\ 
  \varkappa(\bar z_1), \ldots, \varkappa(\bar z_{\hat m_1})\bigr) \Bigr)\varkappa(b_{1 s_1}) \cdot \\ \cdot
   \Bigl( h_{1 s_1} f_1\bigl(\varkappa(a^{(1)}_1),
 \ldots, \varkappa (a^{(1)}_{d_1});
 \ldots;  \varkappa (a^{(1)}_1),
 \ldots, \varkappa (a^{(1)}_{d_1}); 
\\  
 \varkappa(\bar z^{(1)}_1), \ldots,  \varkappa(\bar z^{(1)}_{m_1})\bigr)\Bigr)  \varkappa(\tilde b_{1 s_1} )
 \varkappa(1_{B_1}) \biggr) \Biggr) j_2
 \cdot \\ \cdot
    \prod_{i=2}^{r}\Biggl(\gamma_i\biggl(\varkappa(b_{is_i}) \Bigl(h_{i s_i}  f_i\bigl(\varkappa(a^{(i)}_1),
 \ldots, \varkappa(a^{(i)}_{d_i});
 \ldots;  \varkappa(a^{(i)}_1),
 \ldots, \varkappa(a^{(i)}_{d_i});
 \\
  \varkappa(\bar z^{(i)}_1 ), \ldots, \varkappa(\bar z^{(i)}_{m_i})\bigr)\Bigr) \varkappa(\tilde b_{is_i}) 
  \varkappa(1_{B_i})\biggr) \Biggr) j_{i+1}\ne 0.
 \end{split}
 \end{equation*}

  Define the multilinear function \begin{equation*}\begin{split}
 f_0 := v_1 \Biggl(\gamma_1\biggl(y_0 \Bigl( h_{0 s_1}
\tilde f_1\bigl( x^{(1)}_1,
 \ldots, x^{(1)}_{d_1};
 \ldots;  x^{(2\tilde k)}_1,
 \ldots, x^{(2\tilde k)}_{d_1};
  z_1, \ldots, z_{\hat m_1} \bigr) \Bigr) y_1 \cdot \\ \cdot
   \Bigl( h_{1 s_1} f_1\bigl(x^{(1,1)}_1,
 \ldots, x^{(1,1)}_{d_1};
 \ldots;  x^{(1,2k)}_1,
 \ldots, x^{(1,2k)}_{d_1}; 
 z^{(1)}_1, \ldots,  z^{(1)}_{m_1} \bigr)\Bigr) w_1  \biggr)\Biggr) v_2 
 \cdot \\ \cdot
   \prod_{i=2}^{r}\Biggl(\gamma_i\biggl(y_i \Bigl(h_{i s_i}  f_i\bigl(x^{(i,1)}_1,
 \ldots, x^{(i, 1)}_{d_i};
 \ldots;  x^{(i, 2k)}_1,
 \ldots, x^{(i, 2k)}_{d_i};
  z^{(i)}_1, \ldots, z^{(i)}_{m_i}\bigr)\Bigr) w_i\biggr) \Biggr) v_{i+1}.
 \end{split}
 \end{equation*}   
 (If $q_i=0$ for some $i$ and $j_i$ was absent in~(\ref{EqAssocNonZero}), then the variable $v_i$ does not appear in $f_0$ either.)
The value of $f_0$ under the substitution
$x^{(\alpha)}_{\beta}=\varkappa(a^{(1)}_\beta)$,
    $x^{(i, \alpha)}_{\beta}=\varkappa(a^{(i)}_\beta)$,
    $z_i=\varkappa(\bar z_i)$,
 $z^{(i)}_{\beta}=\varkappa(\bar z^{(i)}_\beta)$, $v_i=j_i$,
 $y_0 = \varkappa(\tilde b_{s_1})$,
 $y_i = \varkappa(b_{is_i})$,
 $w_i = \varkappa(\tilde b_{is_i})\varkappa(1_{B_i})$
  is $a\ne 0$. 
    We denote this substitution by $\Xi$.

Let $X_\ell = \bigcup_{i=1}^r X^{(i,\ell)}$, where $X^{(i,\ell)}=\lbrace x^{(i,\ell)}_\alpha \mid 1\leqslant \alpha \leqslant d_i \rbrace$, and let $\Alt_\ell$
be the operator of alternation on the set $X_\ell$.
   Denote $\hat f := \Alt_1 \Alt_2 \ldots \Alt_{2k} f_0$.
   Note that the alternations do not change $z_i, 
 z^{(i)}_{\beta}, v_i, y_i, w_i$,
   and $f_i$ is alternating on each $X^{(i)}_\ell$.
   Hence the value of $\hat f$ under the substitution $\Xi$
   equals $\left((d_1)! (d_2)! \ldots (d_{r})!\right)^{2k} a \ne 0$
   since $B_1 \oplus \ldots \oplus B_r$ is a direct sum
   of $H$-invariant ideals and if the alternation puts a variable from
   $X^{(i)}_\ell$ on the place of a variable from $X^{(i')}_\ell$
   for $i \ne i'$, the corresponding $h\varkappa(a^{(i)}_\beta)$, $h\in H$, annihilates elements from $\varkappa(B_{i'})$. Here we have used once again that
   by~(\ref{EqAssocbazero}), in the entries to the left of $\varkappa(1_{B_i})$,
the map $\varkappa$ behaves like a homomorphism of $H$-modules and algebras.
  
   Note that without additional manipulations $\hat f$ is a multilinear function but not an $H$-polynomial.
 However, using~(\ref{EqGeneralizedHopf}), we can represent
 $\hat f$  by an $H$-polynomial 
\begin{equation*}\begin{split}
 \tilde f := \Alt_1 \Alt_2 \ldots \Alt_{2k} v_1 y_0^{\tilde h_0} 
\tilde f'_1\left( x^{(1)}_1,
 \ldots, x^{(1)}_{d_1};
 \ldots;  x^{(2\tilde k)}_1,
 \ldots, x^{(2\tilde k)}_{d_1};
  z_1, \ldots, z_{\hat m_1} \right)  y_1^{\tilde h_1} \cdot \\ \cdot
f'_1\left(x^{(1,1)}_1,
 \ldots, x^{(1,1)}_{d_1};
 \ldots;  x^{(1,2k)}_1,
 \ldots, x^{(1,2k)}_{d_1}; 
 z^{(1)}_1, \ldots,  z^{(1)}_{m_1} \right) w_1^{\hat h_1}   v_2 
 \cdot \\ \cdot
   \prod_{i=2}^{r} y_i^{\tilde h_i} f'_i\left(x^{(i,1)}_1,
 \ldots, x^{(i, 1)}_{d_i};
 \ldots;  x^{(i, 2k)}_1,
 \ldots, x^{(i, 2k)}_{d_i};
  z^{(i)}_1, \ldots, z^{(i)}_{m_i}\right) w_i^{\hat h_i}  v_{i+1}
 \end{split}
 \end{equation*}     
where $f_i'$ and $\tilde f_1'$ are some $H$-polynomials,
$\tilde h_i, \hat h_i \in H$ are the elements obtained from $h_{0s_1}$, $h_{is_i}$ and $\gamma_i$ by~(\ref{EqGeneralizedHopf}), and the value of $\tilde f$ under
the substitution $\Xi$ again equals $\left((d_1)! (d_2)! \ldots (d_{r})!\right)^{2k} a \ne 0$.

Now we expand $\tilde f'_1$ and notice that $\tilde f$ is a linear combination of multilinear $H$-polynomials \begin{equation*}\begin{split}
 \tilde f_0 := \Alt_1 \Alt_2 \ldots \Alt_{2k} u_1^{\tau_1} \ldots u_s^{\tau_s}  \cdot \\ \cdot
f'_1\left(x^{(1,1)}_1,
 \ldots, x^{(1,1)}_{d_1};
 \ldots;  x^{(1,2k)}_1,
 \ldots, x^{(1,2k)}_{d_1}; 
 z^{(1)}_1, \ldots,  z^{(1)}_{m_1} \right) w_1^{\hat h_1}   v_2 
 \cdot \\ \cdot
   \prod_{i=2}^{r} y_i^{\tilde h_i} f'_i\left(x^{(i,1)}_1,
 \ldots, x^{(i, 1)}_{d_i};
 \ldots;  x^{(i, 2k)}_1,
 \ldots, x^{(i, 2k)}_{d_i};
  z^{(i)}_1, \ldots, z^{(i)}_{m_i}\right) w_i^{\hat h_i}  v_{i+1}
 \end{split}
 \end{equation*}    
 where $u_1, \ldots, u_s$ are the variables $x^{(\alpha)}_\beta$, 
 $y_0$, $y_1$, $z_i$, and possibly $v_1$, and $\tau_i\in H$ are some elements. Here $s=2\tilde k d_1 + \hat m_1+3$ if $v_1$ was actually present in $f_0$ and $s=2\tilde k d_1 + \hat m_1+2$ either. At least one of $\tilde f_0$ is not a polynomial
 $H$-identity. Denote it again by $\tilde f_0$.
 Note that 
 $$\deg \tilde f_0 \geqslant 2\tilde k d_1 + \hat m_1+1
 + \sum_{i=1}^r (2kd_i+m_i+2)>  2\tilde k d_1 + 2 k d > n.$$
 On the other hand, $\sum_{i=1}^r (2kd_i+m_i+3)-1 \leqslant n$.
 Let \begin{equation*}\begin{split}
 f := \Alt_1 \Alt_2 \ldots \Alt_{2k} u_{(\deg \tilde f_0)-n+1 }^{\tau_{(\deg \tilde f_0)-n+1}} \ldots u_s^{\tau_s}  \cdot \\ \cdot
f'_1\left(x^{(1,1)}_1,
 \ldots, x^{(1,1)}_{d_1};
 \ldots;  x^{(1,2k)}_1,
 \ldots, x^{(1,2k)}_{d_1}; 
 z^{(1)}_1, \ldots,  z^{(1)}_{m_1} \right) w_1^{\hat h_1}   v_2 
 \cdot \\ \cdot
   \prod_{i=2}^{r} y_i^{\tilde h_i} f'_i\left(x^{(i,1)}_1,
 \ldots, x^{(i, 1)}_{d_i};
 \ldots;  x^{(i, 2k)}_1,
 \ldots, x^{(i, 2k)}_{d_i};
  z^{(i)}_1, \ldots, z^{(i)}_{m_i}\right) w_i^{\hat h_i}  v_{i+1}.
 \end{split}
 \end{equation*}    
 Then $f$ does not vanish under the substitution $\Xi$. In addition,
 $f$ is alternating in $X_\ell$, $1 \leqslant \ell \leqslant 2k$.
 Moreover $\deg f = n$. Now we rename the variables of $f$ into $x_1, \ldots, x_n$
 and notice that $f$ satisfies all the conditions of the lemma.
 \end{proof}
\begin{proof}[Proof of Theorem~\ref{TheoremHmodHRadAmitsurPIexpHBdimB}]
As we have already mentioned, the upper bound was proved in~\cite[Lemma 2]{ASGordienko8}. 
 In order to prove the lower bound, we repeat verbatim the proofs of~\cite[Lemma~11 and Theorem~5]{ASGordienko3} using~\cite[Lemma~1]{ASGordienko8} and Lemma~\ref{LemmaAssocLowerPolynomial} instead of~\cite[Lemma~7 and Lemma~10]{ASGordienko3}. 
\end{proof}

\end{document}